\documentclass[11pt]{article}

\usepackage[utf8]{inputenc}
\usepackage[UKenglish]{babel}
\usepackage[T1]{fontenc}
\usepackage{stmaryrd,verbatim}
\usepackage{times}

\usepackage[a4paper,top=1.5cm,bottom=1.5cm,left=2cm,right=2cm]{geometry}

\usepackage{amsmath,amssymb,amsfonts,amsthm,upgreek}
\usepackage{mathrsfs}
\usepackage[pdftex,dvipsnames]{xcolor}
\usepackage[colorinlistoftodos,prependcaption, textsize=tiny, textwidth=2.5cm,color=green]{todonotes}
\usepackage[all]{xy}
\usepackage{cancel}
\usepackage{appendix}
\usepackage{enumerate}
\usepackage{tikz-cd} 

\usepackage{multicol}

\usepackage{changepage}

\usepackage[colorlinks=true, allcolors=blue,backref=page]{hyperref}
\usepackage[initials,alphabetic]{amsrefs}

\newcommand{\rG}{{\rm G}}

\newcommand{\rS}{{\rm S}}
\newcommand{\rT}{{\rm T}}

\newcommand{\bi}{{\bf i}}

\newcommand{\cA}{\mathcal{A}}

\newcommand{\cI}{\mathcal{I}}

\newcommand{\cV}{\mathcal{V}}

\newcommand{\Q}{\mathbb{Q}}
\newcommand{\R}{\mathbb{R}}
\newcommand{\C}{\mathbb{C}}

\newcommand{\SU}{{\rm SU}}

\renewcommand{\P}{\mathbb{P}}

\renewcommand{\epsilon}{\varepsilon}

\newcommand{\Hol}{\mathrm{Hol}}

\renewcommand{\Im}{\mathop{\mathrm{Im}}}

\renewcommand{\Re}{\mathop{\mathrm{Re}}}

\newcommand{\tr}{\mathop{\mathrm{tr}}\nolimits}

\newcommand{\vol}{\mathrm{vol}}

\newcommand{\qandq}{\quad\text{and}\quad}
\newcommand{\qwithq}{\quad\text{with}\quad}
\newcommand{\qforq}{\quad\text{for}\quad}

\def\<{\mathopen{}\left<}
\def\>{\right>\mathclose{}}
\def\({\mathopen{}\left(}
\def\){\right)\mathclose{}}

\usepackage{multicol, color}

\definecolor{gold}{rgb}{0.85,.66,0}
\definecolor{cherry}{rgb}{0.9,.1,.2}
\definecolor{burgundy}{rgb}{0.8,.2,.2}
\definecolor{orangered}{rgb}{0.85,.3,0}
\definecolor{orange}{rgb}{0.85,.4,0}
\definecolor{olive}{rgb}{.45,.4,0}
\definecolor{lime}{rgb}{.6,.9,0}
\definecolor{green}{rgb}{.2,.7,0}
\definecolor{grey}{rgb}{.4,.4,.2}
\definecolor{brown}{rgb}{.4,.3,.1}

\newtheorem{theorem}{Theorem}
\newtheorem{prop}{Proposition}[section]
\newtheorem{cor}[prop]{Corollary}
\newtheorem{lemma}[prop]{Lemma}

\numberwithin{substep}{step}
\newtheorem{case}{Case}

\numberwithin{subcase}{case}

\theoremstyle{remark}
\newtheorem{remark}[prop]{Remark}

\theoremstyle{definition}
\newtheorem{definition}[prop]{Definition}
\newtheorem{example}[prop]{Example}

\newcommand{\ep}{_\epsilon}
\newcommand{\epm}{_{\epsilon,m}}

\author{
  Jason D. Lotay \\
  \emph{University of Oxford}
  \and
  Henrique N. Sá Earp \\
  \emph{University of Campinas (Unicamp)}
  \bigskip
}
\title{The heterotic $\rG_2$ system on contact Calabi--Yau $7$-manifolds}
\date{}

\begin{document}

\maketitle

\begin{abstract}
     We obtain non-trivial approximate solutions to the heterotic $\rm{G}_2$ system on the total spaces of non-trivial circle bundles over Calabi--Yau $3$-orbifolds, which satisfy the equations up to an arbitrarily small error, by adjusting the size of the $S^1$ fibres in proportion to a power of the string constant $\alpha'$.  Each approximate solution provides a cocalibrated $\rm{G}_2$-structure, the torsion of which realises a non-trivial scalar field, a constant (trivial) dilaton  field and an $H$-flux with nontrivial Chern--Simons defect.  The approximate solutions also include a connection  on the tangent bundle which, together with a non-flat $\rm{G}_2$-instanton induced from the horizontal Calabi--Yau metric, satisfy  the anomaly-free condition, also known as the heterotic Bianchi identity.  The approximate solutions fail to be genuine solutions solely because the connections on the tangent bundle are only $\rm{G}_2$-instantons up to higher order corrections in  $\alpha'$.  
\end{abstract}

\begin{adjustwidth}{0.95cm}{0.95cm}
    \tableofcontents
\end{adjustwidth}


\newpage

\section{Introduction}

The heterotic $\rG_2$ system intertwines geometric and gauge-theoretic degrees of freedom over a $7$-manifold with $\rG_2$-structure, subject to instanton-type equations and a prescribed Chern--Simons defect. \textcolor{black}{The latter constraint is required by \textcolor{black}{what physicists refer to as} the Green--Schwartz anomaly cancellation mechanism}. This setup fits in the broader context of so-called Hull--Strominger systems on manifolds with special geometry, particularly in real dimensions 6, 7 and 8, which arise as low-energy effective theories of the heterotic string. Its motivation aside, we should stress from the outset that the language and arguments in this paper are primarily aimed at a mathematical audience.

To the best of our knowledge, the present problem was first formulated in the mathematics literature by Fernandez et al.  \cite{Fernandez2011}, who \textcolor{black}{found} `the first explicit compact valid supersymmetric heterotic solutions
with non-zero flux, non-flat instanton and constant dilaton' on some carefully chosen generalised Heisenberg nilmanifolds. Moreover, they somewhat \textcolor{black}{inspired} our approach, by invoking the methods of Kobayashi \cite{Kobayashi1956} to guarantee, albeit non-constructively, the existence of circle fibrations which \emph{partially} satisfy the heterotic $\rG_2$ system \cite{Fernandez2011}*{Theorem 6.4}. For a comprehensive survey of the problem's origins in the string theory literature, we refer the reader to that paper's Introduction and references therein. 

Over recent years, such Hull--Strominger systems have attracted substantial interest. For instance, Garc\'ia-Fern\'andez et al.~have addressed  description of infinitesimal moduli of solutions to these systems  over a Calabi--Yau \cite{GarciaFernandez2017} or $\rG_2$-manifold \cite{Clarke2016} base, as well as an interpretation of the problem from the perspective of generalised Ricci flow on a Courant algebroid \cite{Garcia-Fernandez2019}. More recently still, Fino et al.~\cite{Fino2021} have found solutions to the Hull--Strominger system in 6 dimensions using 2-torus bundles over K3 orbifolds, extending the fundamental work of Fu--Yau \cite{Fu2008}, which also has some relation to our study.

Our approach to the heterotic $\rG_2$ system will follow most closely the thorough investigation by de la Ossa et al.~in  \cites{delaOssa2016,Ossa2018a,Ossa2018}, who propose, among various contributions, a physically viable formulation of the problem for $\rG_2$-structures with torsion. Indeed, we  study the system \color{black} over so-called contact Calabi--Yau (cCY) $7$-manifolds, which carry cocalibrated $\rG_2$-structures; cCY manifolds were introduced by \cite{Habib2015}, and gauge theory on $7$-dimensional cCY was proposed in \cite{Calvo-Andrade2020} and further studied in \cite{Portilla2019}. Our base $7$-manifolds include the total spaces of $S^1$-(orbi)bundles over every weighted Calabi--Yau $3$-fold famously listed by Candelas-Lynker-Schimmrigk \cite{Candelas1990}, seen as links of isolated hypersurface singularities on $S^9\subset\C^5$.  In particular, we obtain  \emph{constructive} approximate solutions to the heterotic $\rG_2$-system over compact \emph{simply-connected} (actually, $2$-connected) $7$-manifolds as in Example~\ref{ex: CY links}, which can be made arbitrarily close to genuine solutions by shrinking the circle fibres. 

\subsection{\texorpdfstring{Heterotic $\rG_2$ system or $\rG_2$-Hull--Strominger system}{Heterotic G2 system or G2-Hull-Strominger system}}

\begin{definition}
\label{dfn:torsion}
    On a 7-manifold with $\rG_2$-structure $(K^7,\varphi)$, we let $\psi=*\varphi\in\Omega^4(K)$ and recall the following characterisations of some components of $\Omega^\bullet(K)$ corresponding to irreducible $\rG_2$-representations:
\begin{equation*}
\begin{split}
    \Omega^2_{14}(K)&=\{\beta\in\Omega^2(K)\,:\,\beta\wedge\varphi=-*\beta\}=\{\beta\in\Omega^2(K)\,:\,\beta\wedge\psi=0\},\\
    \Omega^3_{27}(K)&=\{\gamma\in\Omega^3(K)\,:\,\gamma\wedge\varphi=0,\,\gamma\wedge\psi=0\}.
\end{split}    
\end{equation*}
The \emph{torsion} of $\varphi$ is completely described by the quantities $\tau_0\in C^{\infty}(K)$, $\tau_1\in\Omega^1(K)$, $\tau_2\in\Omega^2_{14}(K)$ and $\tau_3\in\Omega^3_{27}(K)$, which satisfy
\begin{equation*}
    d\varphi=\tau_0\psi+ 3\tau_1\wedge\varphi+*\tau_3
    \qandq
    d\psi =4\tau_1\wedge\psi+\tau_2\wedge\varphi.
\end{equation*}
\end{definition}

Given a smooth $G$-bundle $F\to K$, for some compact semi-simple Lie group $G$, let $\cA(F)$ denote its space of smooth $G$-connections.

\begin{definition}
\label{dfn:heterotic.G2}
    The \emph{heterotic $\rG_2$ system} or \emph{$\rG_2$-Hull--Strominger system} on a 
    $7$-manifold with $\rG_2$-structure $(K,\varphi)$ is comprised of the following degrees of freedom:
\begin{itemize}
    \item 
    Geometric fields \textcolor{black}{(tensors)}:
    $$
    \lambda\in \R \text{ (scalar field)},
    \quad
    \mu\in C^\infty(K) \text{ (dilaton)},
    \qandq
    H\in\Omega^3(K) \text{ (flux)}.
    $$

    \item
    Gauge fields \textcolor{black}{(connections)}:
    $$
    A\in\cA(E),
    \qandq
    \theta\in\cA(TK),
    $$
    where $E\to K$ is a vector bundle  and both connections are respectively $\rG_2$-instantons:
    $$
    F_A\wedge\psi =0
    \qandq
    R_\theta \wedge \psi=0.
    $$
\end{itemize}
    The geometric fields satisfy the following relations with the torsion of the $\rG_2$-structure $\varphi$:
\begin{equation}
\label{eq:torsion.conditions}
\begin{split}
    \tau_0 &=\frac{3}{7}\lambda\\
    \tau_1 &=\frac{1}{2}d\mu\\
    \tau_2 &=0\\
\end{split}\qquad
\begin{split}
    \textcolor{black}{H^\perp}&\textcolor{black}{=-\frac{1}{2}d\mu^\#\lrcorner\psi-\tau_3}\\
    H&=\frac{\lambda}{14}\varphi\oplus H^\perp\\
    \tau_3 &=-H^\perp-\frac{1}{2}d\mu^\#\lrcorner\psi.
\end{split}
\end{equation}
    Given  a (small) real constant $\alpha'\neq 0$, related to the string scale, the flux compensates exactly the Chern--Simons defect between the gauge fields via the \emph{anomaly-free condition}, also referred to as the \emph{heterotic Bianchi identity}:
\begin{equation}
\label{eq:anomaly}
    dH=\frac{\alpha'}{4}
    \left(
    \tr F_A\wedge F_A-\tr R_\theta\wedge R_\theta
    \right),
\end{equation}
where $F_A$ is the curvature of $A$, $R_\theta$ is the Riemann curvature tensor of $\theta$.
\end{definition}

\begin{remark}\label{rmk:orderalphaprime}
In the physics literature one obtains the heterotic $\rG_2$ system by truncating a system of equations involving (formal) power series in $\alpha'$.  Consequently, one finds statements such as that $\theta$ need only be a $\rG_2$-instanton on $TK$ ``up to $O(\alpha')$-corrections'', cf.~\cite{Ossa2014}*{Appendix B}.  One natural way to formulate the $\rG_2$-instanton condition  to order $O(\alpha')^k$ is
\[
|R_{\theta}\wedge\psi|_{g}=O(\alpha')^k,
\quad\text{as $\alpha'\to 0$,}
\]
where $|.|_g$ is the pointwise $C^0$-norm with respect to the $\rG_2$-metric $g$ defined by  $\varphi=\ast\psi$.  (Note that $\theta$, $\psi$ and $g$ can all depend on $\alpha'$.) We will state our results in Theorem \ref{thm: Main Theorem} below in those terms.  However, we should stress that we still adopt Definition~\ref{dfn:heterotic.G2} for a genuine solution to the heterotic $\rG_2$ system, in accordance with the mathematics literature on Hull--Strominger systems in 6 and 7 dimensions; see e.g.~ \cites{Fernandez2011, Ivanov2010, Clarke2016,GarciaFernandez2016,GarciaFernandez2017,GarciaFernandez2020}.  
\end{remark}

\begin{remark}
For physical reasons one  typically assumes $\alpha'>0$  in \eqref{eq:anomaly}, so we are not interested in the case $dH=0$.  Hence, \eqref{eq:anomaly} only has any hope of occurring under the so-called \emph{omalous} condition:
\begin{equation}
\label{eq:omalous}
    p_1(E)=p_1(K)\in H^4_{dR}(K).
\end{equation}
\end{remark}
\color{black}

Omalous bundles can be systematically constructed for instance via monad techniques, as in the following example, which is derived trivially by combining results from {\cites{Henni2013,Calvo-Andrade2020}}. In this paper, though, we will follow a different approach, cf. Theorem \ref{thm: Main Theorem} below.

\begin{example}
When $K$ is the link in $S^9$ associated to the Fermat quintic $V\subset \P^4$, the cohomology of the monad 
$$
\xymatrix{ 
0\ar[r]&\mathcal{O}_{V}(-1)^{\oplus 10}\ar[r] & \mathcal{O}_{V}^{\oplus 22}\ar[r] & \mathcal{O}_V(1)^{\oplus 10}\ar[r]& 0,
}
$$
is a rank $2$ omalous bundle $E$, i.e.~satisfying \eqref{eq:omalous}, with $c_1=0$ and $c_2=10 $.
\end{example}

\begin{remark}  
Fernandez et al.~\cite{Fernandez2011} argue that one can replace the $\rG_2$-instanton condition on $R_{\theta}$ by a more general second order condition, and still satisfy the equations of motion which motivate the heterotic $\rG_2$ system. However, Ivanov concluded separately that in this context both conditions are equivalent \cite{Ivanov2010}*{\S2.3.1}.
\end{remark}







\subsection{Gauge theory on contact Calabi--Yau (cCY) manifolds}

Let $(M^{2n+1},\eta,\xi)$  denote a contact   manifold, with contact form  $\eta$ and  Reeb vector field $\xi$ \cites{Boyer2008}. 
When $M$ is endowed in addition with a Sasakian structure, namely  an integrable transverse complex structure  $J $ and a compatible metric $g$, Biswas-Schumacher    \cites{Biswas2010} propose a natural  notion of Sasakian holomorphic structure for complex vector bundles $E\to M$. 

We recall that a connection $A$ on a
complex vector bundle over a Kähler manifold is said to be \emph{Hermitian Yang-Mills (HYM)} if
\begin{equation}
\label{eq: tHYM}
    \hat{F}_A := (F_A,\omega)= 0
    \qandq
    F^{ 0,2}_A = 0.
\end{equation}
This notion extends  to Sasakian bundles, by taking $\omega := d\eta\in\Omega^{1,1}(M)$ as a `transverse Kähler form', and defining HYM connections to be the solutions of \eqref{eq: tHYM} in that sense. The well-known concept of \emph{Chern  connection} also extends,  namely as a connection mutually compatible with the holomorphic structure (\emph{integrable}) and a given Hermitian bundle metric (\emph{unitary}), see \cite{Biswas2010}*{\S~3}.

An important class of Sasakian manifolds are those endowed with a \emph{contact Calabi--Yau (cCY)} structure [Definition \ref{def:cCY}], the Riemannian metrics of which have transverse holonomy $\SU(2n+1)$, in the sense of foliations, corresponding to the existence of a global  transverse holomorphic volume form $\Omega\in\Omega^{n,0}(M)$  \cites{Habib2015}. When $n=3$, cCY   $7$-manifolds are naturally endowed with a $\rG_2$-structure defined by the $3$-form
\begin{equation}
    \varphi :=\eta\wedge d\eta+\Re\Omega,
\end{equation} 
which is \emph{cocalibrated}, in the sense that its Hodge dual $\psi:=\ast_g\varphi $ is closed under the de Rham differential. When a $3$-form $\varphi $ on a $7$-manifold defines a $\rG_2$-structure, the condition
\begin{equation}
\label{eq:G2.inst.equation}
    F_A\wedge\psi=0
\end{equation}
is referred to as the  \emph{$\rG_2$-instanton equation}. On holomorphic Sasakian bundles over closed cCY  $7$-manifolds, it has the distinctive feature that integrable solutions \textcolor{black}{(i.e. compatible with the holomorphic Sasakian structure)} are indeed Yang-Mills critical points, even though the $\rG_2$-structure has torsion  \cites{Calvo-Andrade2020}.

\subsection{Statement of main result}


\begin{definition}
\label{dfn:G2.str}
    Let $V$ be a Calabi--Yau 3-orbifold with metric $g_V$, volume form $\vol_V$, K\"ahler form $\omega$ and holomorphic volume form $\Omega$ satisfying
\begin{equation}
\label{eq:volV}
    \vol_V
        =\frac{\omega^3}{3!}
        =\frac{\Re\Omega\wedge\Im\Omega}{4}.
\end{equation}
    Suppose that the total space of $\pi:K\to V$ is a contact Calabi--Yau 7-manifold, i.e. $K$ is a $S^1$-(orbi)bundle, with connection $1$-form $\eta$, such that\footnote{For ease of notation, we omit the pullback $\pi^*$ for forms and tensors defined on $K$ which are pulled back from $V$.}
$d\eta=\omega$.  
    For every $\epsilon>0$, we define a $S^1$-invariant $\rG_2$-structure on $K$ by 
\begin{align}
    \varphi_{\epsilon}
    &=\epsilon\eta\wedge\omega+\Re\Omega,
    \label{eq:phi.eps}\\
    \psi_{\epsilon}
    &=\frac{1}{2}\omega^2-\epsilon\eta\wedge\Im\Omega.
    \label{eq:psi.eps}
\end{align}
    The metric induced from this $\rG_2$-structure and its corresponding volume form are
\begin{equation}
\label{eq:metric.vol.eps}
    g_{\epsilon}
    =\epsilon^2\eta\otimes\eta+g_V
\qandq
    \vol_{\epsilon}
    =\epsilon\eta\wedge \vol_V.
\end{equation}
\end{definition}

\noindent NB.: The choice of $\epsilon>0$ will a posteriori depend on the string parameter $\alpha'$ in  \eqref{eq:anomaly}.

\textcolor{black}{We will see that producing geometric fields satisfying the prescribed relations \eqref{eq:torsion.conditions} with the torsion of the $\rG_2$-structure \eqref{eq:phi.eps} is rather straightforward.} The actual problem consists in obtaining  gauge fields that satisfy the heterotic Bianchi identity \eqref{eq:anomaly} on the contact Calabi--Yau $K^7$. We introduce therefore the following data:
\begin{itemize}
    \item Let $A:=\pi^*\Gamma_V$ be the pullback of the Levi-Civita connection of $g_V$ to  $E:=\pi^*TV\to K$. Then $A$ is a $\rG_2$-instanton on $E$, since it is the pullback of a HYM connection on $TV$ \cite{Calvo-Andrade2020}*{\S4.3}. Moreover, $A$ is a Yang-Mills connection and it minimises the Yang-Mills energy among Chern connections, with respect to the natural Sasakian holomorphic structure of $E$ [ibid., Theorem 1.4].
    
    \item For each fixed $\epsilon>0$, let $\theta_\epsilon$ denote the Levi-Civita connection of the metric $g_\epsilon$ on $K$ of Definition \ref{dfn:G2.str}.  
    It was shown in \cite{Friedrich2003} that there is a unique metric connection which makes $\varphi_\epsilon$ parallel and has totally skew-symmetric torsion (which may be identified with $H_{\epsilon}$), often called the \emph{Bismut connection} (and also sometimes called the \emph{canonical connection}).  Following work in \cite{Hull1986}, another natural metric connection which appears in the physics literature is the \emph{Hull connection}, whose torsion has the opposite sign to the Bismut connection. \color{black}  The \color{black} Bismut and Hull connections fit in a 1-parameter family $\{\theta_\epsilon^{\delta}\}$,  which are modifications of $\theta_\epsilon$ by a prescribed torsion component governed by the parameter $\delta\in\R$ and the flux $H_{\epsilon}$.  We further extend it to a 2-parameter family $\{\theta_\epsilon^{\delta,k}\}$, with\footnote{Choosing $k=0$ would in fact require the $S^1$-fibration $K\to V$ to be trivial, see Remark \ref{rem: k=0 trivial case}.} $k\in\R\smallsetminus\{0\}$, corresponding to ``squashings'' of the connections $\theta_{\epsilon}^{\delta}$.  Finally, we define a ``twist'' by an additional parameter $m\in\R$, to obtain our overall family of connections $\{\theta_{\epsilon,m}^{\delta,k}\}$ on $TK$ [Proposition \ref{prop:twist.curvature}].
    Whilst typically $\theta_{\epsilon,m}^{\delta,k}$ will \emph{not} be a $\rG_2$-instanton on $TK$, it does satisfy the $\rG_2$-instanton condition up to $O(\alpha')$-corrections (in the sense of Remark \ref{rmk:orderalphaprime}) \color{black} for various parameter choices. 
\end{itemize}

\begin{theorem}
\label{thm: Main Theorem}
    Let $(K^{7},\eta,\xi,J ,\Omega)$ be a contact Calabi--Yau $7$-manifold, fibering by $\pi:K^{7}\to V$ over the Calabi--Yau $3$-fold $(V,g_V,\omega,J,\Omega)$, and let $E:=\pi^*TV\to K$.
    
    Given any $\alpha'>0$, 
    there exist $k(\alpha'),\epsilon(\alpha')>0$ and $m,\delta\in\R$ such that
    the following assertions hold:
\begin{enumerate}[(i)]
    \item 
    The $\rG_2$-structure \eqref{eq:phi.eps} is coclosed and satisfies the torsion conditions \eqref{eq:torsion.conditions}, with scalar field $\lambda=\frac{\epsilon}{2}$, constant dilaton  $\mu\in \R$, and flux $H_\epsilon= -\epsilon^2\eta\wedge\omega +\epsilon\Re\Omega$.
        
    \item
    The connection $A:=\pi^*\Gamma_V$ is a $\rG_2$-instanton on $E$, with respect to the dual $4$-form \eqref{eq:psi.eps}.
        
    \item
    There exists a 
    connection $\theta:=\theta_{\epsilon,m}^{\delta,k}$ on $TK$, with torsion 
    $$
        H^{\delta,k}_{\epsilon,m} =\big(1-k-\frac{km}{2}\big)\epsilon^2\omega\otimes \eta+\frac{km\epsilon^2}{2}\eta\wedge\omega+k\delta H_{\epsilon},
    $$ 
    which satisfies the $\rG_2$-instanton condition \eqref{eq:G2.inst.equation} to order $O(\alpha')^2$ with respect to the dual $4$-form \eqref{eq:psi.eps}; i.e.
    \begin{equation}
    \label{eq:order.alpha}
    |R_{\theta}\wedge \psi_{\epsilon}|_{g_{\epsilon}}=O(\alpha')^2\quad\text{as $\alpha'\to 0$.}
    \end{equation}
    \item
    The data $(H_\epsilon,A,\theta)$ satisfy the heterotic Bianchi identity \eqref{eq:anomaly}:
    \begin{equation}
    \label{eq: dH=trF^2-trR^2}
        dH_{\epsilon}
        =\frac{\alpha'}{4} (\tr F_A^2-\tr R_\theta^2).
    \end{equation}
        
    \item
    $\displaystyle\lim_{\alpha'\to0} \epsilon(\alpha')=0 \qandq \lim_{\alpha'\to0} k(\alpha')=\infty$.
\end{enumerate}
\end{theorem}

The various components of the proof are developed throughout the paper, and aggregated in \S\ref{sec: proof of main theorem}. 

\bigskip

\noindent\textbf{Acknowledgements:}
The present article stems from an ongoing Newton Mobility  bilateral collaboration (2019-2021), granted by the UK Royal Society [NMG$\backslash$R1$\backslash$191068].  JDL is also partially supported by the Simons Collaboration on Special Holonomy in Geometry, Analysis, and Physics ($\#724071$ Jason Lotay). HSE has also been funded by the São Paulo Research Foundation (Fapesp)  \mbox{[2017/20007-0]} \& \mbox{[2018/21391-1]} and the Brazilian National Council for Scientific and Technological Development (CNPq)  \mbox{[307217/2017-5]}.
The authors would like to thank Xenia de la Ossa, Eirik Svanes and Mario Garc\'ia-Fern\'andez for several insightful discussions.

\section{Contact Calabi--Yau geometry: scalar field, dilaton, and flux}

One may interpret special structure group reductions on compact odd-dimensional Riemannian manifolds as `transverse even-dimensional' structures with respect to a  $S^1$-action. So for instance contact geometry may be seen as transverse symplectic geometry, almost-contact geometry  as tranverse almost-complex geometry, and in the same way Sasakian geometry as transverse Kähler geometry. In particular, one may consider reduction of the transverse holonomy group; indeed Sasakian manifolds with transverse holonomy $\SU(n)$ are studied by Habib and Vezzoni \cite{Habib2015}*{\S~6.2.1}:
\begin{definition}
\label{def:cCY}
    A Sasakian manifold $(K^{2n+1},\eta,\xi,J,\Omega)$ is said to be a \emph{contact Calabi--Yau manifold} (cCY) if $\Omega$ is a nowhere-vanishing transverse form of horizontal type $(n,0)$, such that 
$$
\Omega\wedge\bar{\Omega}
=(-1)^{\frac{n(n+1)}{2}}\bi^{n}\omega^n
\qandq 
d\Omega=0,
\qwithq
\omega=d\eta.
$$
\end{definition} 

Let us specialise to real dimension $7$. It is well-known that, for a Calabi--Yau 3-fold $(V, \omega, \Omega)$, the product $V\times \rS^1$  has a natural torsion-free  $\rG_2$-structure defined by: $\varphi:= dt\wedge\omega+\Re\Omega,$
where $ t$ is the coordinate on $\rS^1$. The Hodge dual of $\varphi$ is 
\begin{equation}
\label{eq:psi cCY}
    \psi:=\ast\varphi
    =\frac{1}{2}\omega\wedge\omega +dt\wedge\Im\Omega
\end{equation}
and the induced metric $g_\varphi=g_V + dt\otimes dt $ is the Riemannian product metric on $V\times S^1$ with holonomy $\Hol(g_\varphi)=\SU(3)\subset\rG_2$. A contact Calabi--Yau structure essentially emulates all of these features, albeit its $\rG_2$-structure has some symmetric torsion. 

\begin{prop}[{\cite{Habib2015}*{\S6.2.1}}] \label{prop:G2estruturaCCY} 
 Every cCY manifold $(K^7,\eta,\xi,J ,\Omega)$ is an $S^1$-bundle $\pi:K\to V$ over a Calabi--Yau 3-orbifold $(V,\omega,\Omega)$, with connection 1-form $\eta$ and curvature  
\begin{equation}
\label{eq:eta.curv}
    d\eta=\omega,    
\end{equation} and it carries a cocalibrated $\rG_2$-structure
\begin{equation}
\label{eq:G2structure}
    \varphi 
    :=\eta\wedge \omega +\Re\Omega,
\end{equation} 
with torsion $d\varphi= \omega\wedge\omega$   and Hodge dual  $ 4$-form
$ \psi=\ast\varphi = \frac{1}{2}\omega\wedge\omega+ \eta\wedge\Im\Omega$.
\end{prop}

\begin{example}[Calabi--Yau links for $k=1$]
\label{ex: CY links}
    Given a rational weight vector $w=(w_0,\dots,w_4)\in\Q^{5}$, a $w$-weighted homogeneous polynomial $f\in\C[z_{0}, \dots, z_{4}]$ of degree $d=\sum_{i=0}^4 w_i$ cuts out an affine hypersurface $\cV=(f)$ with an isolated singularity at $0\in\C^5$. 
\begin{multicols}{2}
    Its link \textcolor{black}{$K_f:=\cV\cap S^9\subset\C^5$} on a local $9$-sphere is a compact and $2$-connected smooth cCY $7$-manifold, fibering by circles over a Calabi--Yau $3$-orbifold $V\subset \P^4(w)$ by the weighted Hopf fibration \cite{Calvo-Andrade2020}*{Theorem 1.1}:
\columnbreak
    \[\begin{tikzcd}        
    K_{f}^7 \arrow[hook]{r} \arrow{d}& S^{9} \arrow{d}\\
        V \arrow[hook]{r}& \P^{4}(w)
    \end{tikzcd}\]
\end{multicols}
In particular, $V$ can be assumed to be any of the weighted Calabi--Yau $3$-folds listed by Candelas-Lynker-Schimmrigk \cite{Candelas1990}. For a detailed survey on Calabi--Yau links, see \cite{Calvo-Andrade2020}*{\S2}.
The $\C$-family of Fermat quintics yields but the simplest of instances, and indeed the only one for which the base $V$ is smooth.
\end{example}

\subsection{Torsion forms and flux of the $\rG_2$-structure $\varphi\ep$}

We begin by addressing the heterotic $\rG_2$ system conditions \eqref{eq:torsion.conditions} on the $\rG_2$-structure, as prescribed by \cite{delaOssa2016}.  In particular, we identify the components of the torsion corresponding to the scalar field, the dilaton  and the flux, as asserted in Theorem \ref{thm: Main Theorem}--(i).

We see from  \eqref{eq:phi.eps}, \eqref{eq:psi.eps}, \eqref{eq:eta.curv}, and the fact that $V$ is Calabi--Yau, that
\begin{equation}
\label{eq:diff.eps}
    d\varphi_{\epsilon}
    =\epsilon  \omega^2
    \qandq 
    d\psi_{\epsilon}=0,
\end{equation}
so that the $\rG_2$-structures of Definition \ref{dfn:G2.str} are coclosed. We can now compute their torsion forms.

\begin{lemma}
\label{lem:torsion.eps}
    For each $\epsilon>0$, the $\rG_2$-structure on $K^7$ defined by \eqref{eq:phi.eps}--\eqref{eq:psi.eps} has torsion forms
    $$
    \begin{array}{ll}
        \tau_0=\displaystyle
        \frac{6}{7}\epsilon ,
        & \tau_1=0,\\
        \tau_2=0,
        &\tau_3=\displaystyle
        \frac{8}{7}\epsilon^2\eta \wedge \omega-\frac{6}{7}\epsilon \Re \Omega.
    \end{array}.
    $$
\end{lemma}
\begin{proof}
    The fact that $\tau_1$ and $\tau_2$ vanish is an immediate consequence of \eqref{eq:diff.eps}.
    Again by \eqref{eq:diff.eps} and definition of the torsion forms, we have:
    \begin{equation}
    \label{eq:diff.phi.eps}
        d\varphi_{\epsilon}
        =\epsilon  \omega^2
        =\tau_0\psi_\epsilon+*_\epsilon\tau_3.
    \end{equation}
    Thus, using $\omega\wedge\Omega=0$, we find
    \begin{equation}
    \label{eq:tau0.eps.vol}
        7\tau_0\vol_{\epsilon}
        =d\varphi_\epsilon \wedge \varphi_\epsilon 
        =\epsilon  \omega^2\wedge (\epsilon\eta\wedge\omega)
        =6\epsilon (\epsilon\eta\wedge\frac{\omega^3}{3!}).
    \end{equation}
    We further deduce from \eqref{eq:tau0.eps.vol}  and the expression of volume form \eqref{eq:metric.vol.eps}  that
    \begin{equation}
    \label{eq:tau0.eps}
        \tau_0=\frac{6}{7}\epsilon.
    \end{equation}
    Moreover, substituting \eqref{eq:tau0.eps} into \eqref{eq:diff.phi.eps}, we see that
    \begin{align}
    \label{eq:star.tau3}
        *_{\epsilon}\tau_3
        &=d\varphi_\epsilon -\tau_0\psi_{\epsilon}
        =\epsilon  \omega^2 -\frac{6}{7}\epsilon (\frac{1}{2}\omega^2
        -\epsilon\eta \wedge\Im\Omega)
        =\frac{4}{7}\epsilon \omega^2 +\frac{6}{7}\epsilon^2 \eta\wedge\Im\Omega.
    \end{align}
    Therefore, using \eqref{eq:metric.vol.eps} and \eqref{eq:star.tau3} we obtain
    \begin{align*}
        \tau_3
        &=\frac{8}{7}\epsilon \ast_{\epsilon} (\frac{1}{2}\omega^2) +\frac{6}{7} \epsilon  \ast_{\epsilon}(\epsilon\eta\wedge\Im\Omega)
        =\frac{8}{7}\epsilon^2  \eta\wedge\omega -\frac{6}{7}\epsilon \Re\Omega.
        \qedhere
    \end{align*}
\end{proof}

We may compute the flux of the $\rG_2$ structure $\varphi_{\epsilon}$ as follows.

\begin{lemma}
\label{lem:flux}
    In the situation of Lemma \ref{lem:torsion.eps}, the flux of the $\rG_2$ structure $\varphi_{\epsilon}$ is

\begin{equation}
\label{eq:flux}
    H_\epsilon=-\epsilon^2  \eta\wedge\omega+\epsilon  \Re\Omega.
\end{equation}
Hence, 
\begin{equation}\label{eq:flux.deriv}
    dH_{\epsilon}=-\epsilon^2\omega^2.
\end{equation}
\end{lemma}
\begin{proof}
From Definition \ref{dfn:heterotic.G2} and the Lemma, we compute directly:
\begin{align*}
    H_{\epsilon}
    &=\frac{\lambda}{14}\varphi_{\epsilon} +(H_{\epsilon})^\perp
    =\frac{\tau_0}{6}\varphi_{\epsilon}-\tau_3\\
    &=\frac{1}{7}\epsilon(\epsilon\eta\wedge\omega+\Re\Omega)-(\frac{8}{7}\epsilon^2\eta\wedge\omega-\frac{6}{7}\epsilon\Re\Omega)\\
    &=-\epsilon^2\eta\wedge\omega+\epsilon\Re\Omega.
    \qedhere
\end{align*}
\end{proof}

\subsection{Local orthonormal coframe}

One key strategy in our construction consists in varying the length of the $S^1$-fibres on 
$K$ as a function of the string parameter $\alpha'$. With that in mind, we adopt a useful local orthonormal coframe  as follows.

\begin{definition}
\label{dfn:coframe}  
    Given $\epsilon>0$, let $(K^7,\varphi\ep)$ be as in Definition \ref{dfn:G2.str}.
    We choose the local Sasakian real orthonormal coframe on $K$:
\begin{equation}\label{eq:coframe}
    e_0=\epsilon\eta,\quad e_1,\quad e_2,\quad e_3, \quad Je_1, \quad Je_2, \quad Je_3,
\end{equation}
where $J$ is the transverse complex structure (from the Calabi--Yau 3-fold $V$) acting on 1-forms, and we have a basic $\SU(3)$-coframe $\{e_1,e_2,e_3,Je_1,Je_2,Je_3\}$, the pullback of an $\SU(3)$-coframe on $V$, such that
\begin{align}
    \omega
    &=e_1\wedge Je_1+e_2\wedge Je_2+e_3\wedge Je_3,
    \label{eq:omega.coframe}\\
    \Omega
    &=(e_1+iJe_1)\wedge (e_2+iJe_2)\wedge (e_3+iJe_3).
    \label{eq:Omega.coframe}
\end{align}
\end{definition}

\begin{remark}It is worth noting from \eqref{eq:Omega.coframe} that
\begin{align}
    \Re\Omega &=e_1\wedge e_2\wedge e_3-e_1\wedge Je_2\wedge Je_3
    -e_2\wedge Je_3\wedge Je_1-e_3\wedge Je_1\wedge Je_2,\label{eq:Re.Omega.coframe}\\
    \Im\Omega&=Je_1\wedge e_2\wedge e_3+Je_2\wedge e_3\wedge e_1+Je_3\wedge e_1\wedge e_2-Je_1\wedge Je_2\wedge  Je_3.\label{eq:Im.Omega.coframe}
\end{align}
\end{remark}

Using \eqref{eq:omega.coframe} and \eqref{eq:Im.Omega.coframe}, we easily derive the precise expression of $\psi_\epsilon$ in this frame: 
\begin{align}
\label{eq:psi.eps.frame}
\begin{split}
    \psi_{\epsilon}=\frac{1}{2}\omega^2-\epsilon\eta\wedge\Im\Omega&=e_2\wedge Je_2\wedge e_3\wedge Je_3+e_3\wedge Je_3\wedge e_1\wedge Je_1+e_1\wedge Je_1\wedge e_2\wedge Je_2\\
    &\quad-e_0\wedge(Je_1\wedge e_2\wedge e_3+Je_2\wedge e_3\wedge e_1+Je_3\wedge e_1\wedge e_2-Je_1\wedge Je_2\wedge  Je_3).
\end{split}
\end{align}

\begin{lemma}
\label{lem:basic}
    In terms of the local coframe \eqref{eq:coframe} and the natural matrix operations described in Appendix \ref{sec: appendix - matrix operations}, the basic $3$-covectors $e=(e_1,e_2,e_3)^\rT$ and $Je=(Je_1,Je_2,Je_3)^\rT$ in $\Omega^1(K)^{\oplus3}$ have the following properties.
\begin{itemize}
    \item[(a)] The vectors
\begin{equation*}
    e\times Je\quad\text{and}\quad e\times e-Je\times Je
\end{equation*}
consist of basic forms of type $(2,0)+(0,2)$.
\item[(b)] The vector
\begin{equation*}
    e\times e+ Je\times Je
\end{equation*}
and the off-diagonal part of
\begin{equation}\label{eq:off.diag}
    [e]\wedge [Je]-[Je]\wedge [e]
\end{equation}
consist of basic forms of type $(1,1)$ which are also primitive (i.e.~wedge with $\omega^2$ to give zero).  The diagonal part of \eqref{eq:off.diag} consists of basic forms of type $(1,1)$.
\end{itemize}
\end{lemma}
\begin{proof}
For (a), we notice that 
\begin{align*}
    e_2\wedge Je_3-e_3\wedge Je_2&=\Im((e_2+iJe_2)\wedge(e_3+iJe_3)),\\
    e_2\wedge e_3-Je_2\wedge Je_3&=\Re((e_2+iJe_2)\wedge (e_3+iJe_3)).
\end{align*}
We deduce that $e\times Je$ and $e\times e-Je\times Je$ consist of basic forms of type $(2,0)+(0,2)$ as claimed.

For (b), we observe that 
\begin{equation*}
e_2\wedge e_3+Je_2\wedge Je_3=\Re((e_2+iJe_2)\wedge (e_3-iJe_3)),    
\end{equation*}
and hence $e\times e+Je\times Je$ consists of primitive forms of basic type $(1,1)$.  We now note that
\begin{equation}\label{eq:box.e.Je.identity}
    [e]\wedge [Je]-[Je]\wedge [e]=e\wedge Je^{\rm T}-Je\wedge e^{\rm T}-2\omega I
\end{equation}
by Lemma \ref{lem: computational rules}.  Since
\begin{align*}
    e_2\wedge Je_3+e_3\wedge Je_2
    =\Im((e_2-iJe_2)\wedge (e_3+iJe_3)),
\end{align*}
we deduce that the off-diagonal part of $[e]\wedge [Je]-[Je]\wedge [e]$ consists of forms of basic type $(1,1)$ which are primitive also.  Finally, we now see from \eqref{eq:box.e.Je.identity} that the diagonal entries in $[e]\wedge [Je]-[Je]\wedge [e]$ define the diagonal matrix
\begin{equation}\label{eq:diag}
    -2\,\text{diag}(e_2\wedge Je_2+e_3\wedge Je_3,e_3\wedge Je_3+e_1\wedge Je_1,e_1\wedge Je_1+e_2\wedge Je_2),
\end{equation}
which clearly consists of basic forms of type $(1,1)$.
\end{proof}

\section{Gauge fields: \texorpdfstring{G\textsubscript{2}}{G2}-instanton, Bismut, Hull and twisted connections}

It is well-known that the pullback of a basic HYM connection to the total space of a contact Calabi--Yau (cCY) $7$-manifold is a $\rG_2$-instanton, with respect to the standard $\rG_2$-structure \cite{Calvo-Andrade2020}*{\S4.3}. Since the Levi-Civita connection of the Calabi--Yau $(V,g_V)$ on $TV$ is HYM, the following result establishes Theorem \ref{thm: Main Theorem}--(ii).

\begin{lemma}\label{lem:A} Let $E=\pi^*TV$ be the pullback of $TV$ to $K$ via the projection $\pi:K\to V$.  Let $A$ be the connection on $E$ given by the pullback of the Levi-Civita connection of $g_V$.  Then $A$ is a $\rG_2$-instanton on $E$ with holonomy contained in $\mathrm{SU}(3)$.
\end{lemma}

In this section we give formulae for the connections $\theta_{\epsilon,m}^{\delta,k}$ and $A$ and their curvatures with respect to the local coframe in Definition \ref{dfn:coframe}.  Using the freedom given by all three parameters, we will show that $\theta_{\epsilon,m}^{\delta,k}$ can be chosen to satisfy the $\rG_2$-instanton condition, at least to higher orders of the string scale $\alpha'$.

\subsection{The \texorpdfstring{G\textsubscript{2}}{G2}-instanton \texorpdfstring{$A$}{A} and the ``squashings'' \texorpdfstring{$\theta_{\epsilon}^{k}$}{theta^k_epsilon} of the Levi-Civita connection 
}

\subsubsection{Local connection matrices 
}

Since the choice of a local Sasakian coframe on $K$ naturally trivialises $E=\pi^*TV \hookrightarrow TK$, we now want to relate the local matrix of the Levi-Civita connection $\theta_{\epsilon}$ on $TK$ to (the pullback of) the gauge field $A$.  To that end, we compute the first structure equations of our natural coframe:
\begin{prop}
\label{prop:str.eqs} 
    The coframe \eqref{eq:coframe} on $K$  satisfies the following structure equations:
    \begin{align}
        de_0 &=\epsilon\omega=\epsilon(e_1\wedge Je_1+e_2\wedge Je_2+e_3\wedge Je_3),\label{eq:str.eqs.1}\\
        de_i&=-a_{ij}\wedge e_j-b_{ij}\wedge Je_j,\label{eq:str.eqs.1.e}\\
        d(Je_i)&=b_{ij}\wedge e_j-a_{ij}\wedge Je_j,\label{eq:str.eqs.1.Je}
    \end{align}
    for some local 1-forms $a_{ij},b_{ij}$, using the summation convention, with $1\leq i,j\leq3$.  Moreover,
    \begin{equation}
        a_{ji}=-a_{ij},\quad b_{ji}=b_{ij},\quad \sum_{i=1}^3b_{ii}=0,
    \end{equation}
    so the matrix  $a:=(a_{ij})$ is skew-symmetric, and the matrix $b:=(b_{ij})$ is symmetric traceless. Letting  $I:=(\delta_{ij})$ and $e:=(e_1\; e_2\; e_3)^{\rm T}$, the structure equations \eqref{eq:str.eqs.1}--\eqref{eq:str.eqs.1.Je} can be written in terms of $7\times7$ matrices: 
    \begin{equation}
    \label{eq:str.eqs.matrix}
    d\left(\begin{array}{c} 
        e_0\\ e\\ Je
    \end{array}\right)
    =
    -\left(\begin{array}{ccc} 
        0 & \frac{\epsilon}{2}Je^{\rm T} & -\frac{\epsilon}{2}e^{\rm T}\\
        -\frac{\epsilon}{2}Je & a & b-\frac{\epsilon}{2}e_0I\\
        \frac{\epsilon}{2}e & -b+\frac{\epsilon}{2}e_0I & a
    \end{array}\right)
    \wedge 
    \left(\begin{array}{c} 
        e_0\\ e\\ Je
    \end{array}\right).
\end{equation}
\end{prop}
\begin{proof}
The first equation \eqref{eq:str.eqs.1} is a direct consequence of \eqref{eq:coframe} and \eqref{eq:omega.coframe}.
The relationship between the derivatives of $e_i$ and $Je_i$ and the properties of the $a_{ij}$ and $b_{ij}$  are a consequence of $J$ being covariantly constant (on $V$) and $A$ 
having holonomy contained in $\SU(3)$, since $A$ arises from a torsion-free $\SU(3)$-structure.  
\end{proof}


It will be useful later to have the following corollary of the structure equations, which is an elementary computation using \eqref{eq:str.eqs.matrix}.

\begin{prop}\label{prop:str.eqs.2}
Using the notation of Definition \ref{dfn:box.vector}, the coframe in Definition \ref{dfn:coframe} satisfies
\begin{equation}\label{eq:str.eqs.2}
\begin{split}
d([e])&=-a\wedge [e]-[e]\wedge a+b\wedge [Je]-[Je]\wedge b,\\
d([Je])&=-a\wedge [Je]-[Je]\wedge a-b\wedge [e]+[e]\wedge b.
\end{split}    
\end{equation}
\end{prop}

The matrix in \eqref{eq:str.eqs.matrix} represents the Levi-Civita connection $\theta_{\epsilon}$ in the given local coframe, and setting $\epsilon=0$ in that matrix gives the matrix of $A$.  Hence, we have the following.
\begin{cor}
\label{cor:conn.eps} 
    If we let 
\begin{equation}
\label{eq:A}
    A=\left(\begin{array}{ccc} 
        0 & 0 & 0\\
        0 & a & b\\
        0 & -b & a
    \end{array}\right)
    \end{equation}
    and
\begin{equation}
\label{eq:B}
    B= 
    \left(\begin{array}{ccc} 
        0 & Je^{\rm T} & -e^{\rm T}\\
        -Je &0 & -e_0I\\
        e & e_0I & 0
    \end{array}\right),
\end{equation}
    then the Levi-Civita connection $\theta_{\epsilon}$ of the metric $g_{\epsilon}$ in \eqref{eq:metric.vol.eps} is given locally by
\begin{align*}
    \theta_{\epsilon}
    &=\left(\begin{array}{ccc} 
        0 & \frac{\epsilon}{2}Je^{\rm T} & -\frac{\epsilon}{2}e^{\rm T}\\
        -\frac{\epsilon}{2}Je & a & b-\frac{\epsilon}{2}e_0I\\
        \frac{\epsilon}{2}e & -b+\frac{\epsilon}{2}e_0I & a
    \end{array}\right)
    =A+\frac{\epsilon}{2}B.
\end{align*}
\end{cor}

Corollary \ref{cor:conn.eps} allows us to define a family of connections $\theta^{k}_{\epsilon}$ on $TK$ as follows.

\begin{definition}
\label{dfn:theta.eps.k}
    For each $0\neq k\in\R$, let $\theta^{k}_{\epsilon}$ be the connection on $TK$ given, in the local coframe of Definition \ref{dfn:coframe}, by
\begin{equation*}
    \theta^k_{\epsilon}
    :=A+\frac{k\epsilon}{2}B,
\end{equation*}
    with $A$ and $B$ as in Corollary \ref{cor:conn.eps}.  
\end{definition}

\begin{remark}
\label{rem: k=0 trivial case}
    The \emph{trivial case} $k=0$ can only occur when $K=S^1\times V$ is a trivial bundle over $V$, and then the connection on $TK$ will be equal to the pullback of the Levi-Civita connection on $V$ (trivial along $S^1$).  Since we are assuming that $K\to V$ is a non-trivial $S^1$-bundle, we require $k\neq 0$. 
\end{remark}
\begin{remark}
    Notice that
\begin{align*}
    d\left(\begin{array}{c} e_0\\ e\\ Je\end{array}\right)
    &=-\left(A+\frac{k\epsilon}{2}B\right)\wedge \left(\begin{array}{c} e_0\\ e\\ Je\end{array}\right)+\frac{(k-1)\epsilon}{2}B\wedge \left(\begin{array}{c} e_0\\ e\\ Je\end{array}\right)\\
    &=-\theta^k_{\epsilon}\wedge \left(\begin{array}{c} e_0\\ e\\ Je\end{array}\right)+(1-k)\epsilon\left(\begin{array}{c}\omega\\ 0\\ 0\end{array}\right).
\end{align*}
    Therefore, we may view $\theta^k_{\epsilon}$ as a metric connection  on $K$, with torsion $(1-k)\epsilon\omega\otimes e_0$.
    \textcolor{black}{Since $k\neq 0$, we see from Corollary \ref{cor:conn.eps} and Definition \ref{dfn:theta.eps.k} that we may view $\theta^k_{\epsilon}$ as a ``squashing'' of the Levi-Civita connection $\theta_{\epsilon}$ of the metric $g_{\epsilon}$ on $K$.}
\end{remark}

\subsubsection{Local curvature matrices
}

We begin by relating the curvature of the connections $\theta^k_\epsilon$ in Definition \ref{dfn:theta.eps.k} to the curvature $F_A$ of $A$.

\begin{prop}
    In the local coframe of Definition \ref{dfn:coframe}, the curvature $R_{\theta_{\epsilon}^{k}}$ of the  connection $\theta_{\epsilon}^{k}$ from Definition \ref{dfn:theta.eps.k}   satisfies
\begin{equation*}
    R_{\theta_{\epsilon}^{k}}=F_A+\frac{k\epsilon^2}{2}\omega\mathcal{I}+\frac{k^2\epsilon^2}{4}B\wedge B,
\end{equation*}
where
\begin{equation}\label{eq:mathcal.I}
\mathcal{I}=\left(\begin{array}{ccc} 0 & 0 & 0 \\
0 & 0 & -I \\
0 & I & 0\end{array}\right)
\end{equation}
and
\begin{equation}
\begin{split}
\label{eq:BwedgeB}
    B\wedge B
    &=\left(\begin{array}{ccc}
        0 &e_0\wedge e^{\rm T} & e_0\wedge Je^{\rm T} \\
        -e_0\wedge e & -Je\wedge Je^{\rm T} & Je \wedge e^{\rm T}\\
        -e_0\wedge Je & e\wedge Je^{\rm T} & -e\wedge e^{\rm T}
    \end{array}\right)\\
    &=e_0
    \wedge\left(\begin{array}{ccc} 0 & e^{\rm T} & Je^{\rm T}\\ 
    -e & 0 & 0\\
    -Je & 0 & 0
    \end{array}\right)+
    \left(\begin{array}{ccc}
        0 & 0 & 0 \\
        0 & -Je\wedge Je^{\rm T} & Je \wedge e^{\rm T}\\
        0 & e\wedge Je^{\rm T} & -e\wedge e^{\rm T}
    \end{array}\right).
    \end{split}
    \end{equation}
\end{prop}

\begin{proof}
From the relation between $\theta_{\epsilon}^{k}$ and $A$ in Corollary \ref{cor:conn.eps}, we see that
\begin{align}
    R_{\theta_{\epsilon}^{k}}
    &=d\theta_{\epsilon}^{k} +\theta_{\epsilon}^{k}\wedge\theta_{\epsilon}^{k}
    \nonumber\\
    &=dA+\frac{k\epsilon}{2}dB+(A+\frac{k\epsilon}{2}B)\wedge(A+\frac{k\epsilon}{2}B)
    \nonumber\\
    &=F_A+\frac{k\epsilon}{2}(dB+A\wedge B+B\wedge A)+\frac{k^2\epsilon^2}{4}B\wedge B.
    \label{eq:R.eps.k.1}
\end{align}
The first term of interest in \eqref{eq:R.eps.k.1} is 
\begin{align}
    &dB +A\wedge B+B\wedge A\nonumber\\
    &=\left(\begin{array}{ccc}
        0 & d(Je^{\rm T}) +Je^{\rm T}\wedge a+e^{\rm T}\wedge b & -d(e^{\rm T})+Je^{\rm T}\wedge b-e^{\rm T}\wedge a\\
        -d(Je)-a\wedge Je +b\wedge e & b\wedge e_0 I+e_0 I\wedge b & -d(e_0)I -a\wedge e_0 I-e_0 I\wedge a\\
        d(e)+b\wedge Je+a\wedge e & d(e_0 )I+a\wedge e_0 I+e_0 I\wedge a & b\wedge e_0 I+e_0 I\wedge b
    \end{array}\right)
    \nonumber\\
    &=\epsilon\omega 
    \left(\begin{array}{ccc} 
        0 & 0 & 0 \\
        0 & 0 & -I \\
        0 & I & 0
    \end{array}\right)
    =\epsilon\omega\mathcal{I}
    \label{eq:dAB}
\end{align}
as a consequence of the structure equations for the coframe in Proposition \ref{prop:str.eqs}. Equation \eqref{eq:BwedgeB} follows directly from \eqref{eq:B}.  
\end{proof}

At this point, it is worth recalling that $A$ is a $\rG_2$-instanton, in fact the lift of a connection with holonomy $\mathrm{SU}(3)$ on $V$, so $F_A$ must take values in $\mathfrak{su}(3)\subseteq\mathfrak{g}_2$:
\begin{equation}
\label{eq:FA}
    F_A=
    \left(\begin{array}{ccc} 
        0 & 0 & 0 \\ 
        0 & \alpha & \beta \\
        0 &-\beta & \alpha
    \end{array}\right),
\end{equation}
where $\alpha$ is a skew-symmetric $3\times 3$ matrix of 2-forms, and $\beta$ is a symmetric traceless $3\times 3$ matrix of 2-forms. 

\begin{lemma}
    The block-elements of the curvature matrix \eqref{eq:FA} of $A$ in the local coframe \eqref{eq:coframe}, satisfy: 
    \begin{align}
    \label{eq: F ^ e and Je}
    \begin{split}
        \alpha\wedge e+\beta\wedge Je&=0,\\
        \alpha\wedge Je-\beta\wedge e&=0.
    \end{split}
    \end{align}
Moreover, using the notation of Definition \ref{dfn:box.vector}, we have
\begin{equation}\label{eq:F.wedge.e.Je.2}
    \begin{split}
    \alpha\wedge[e]-[e]\wedge \alpha-\beta\wedge[Je]-[Je]\wedge\beta&=0,\\
    \alpha\wedge [Je]+[Je]\wedge\alpha+\beta\wedge[e]-[e]\wedge\beta.
    \end{split}
\end{equation}
\end{lemma}
\begin{proof}
Differentiating the defining relation \begin{equation*}
    d\left(\begin{array}{c} 0\\ e\\ Je \end{array}\right)=-A\wedge \left(\begin{array}{c} 0\\ e\\ Je \end{array}\right),
\end{equation*}
we obtain 
\begin{align*}
    0&=-d A\wedge 
    \left(\begin{array}{c}
        0\\ e\\ Je 
    \end{array}\right) +A\wedge  d\left(\begin{array}{c} 
        0\\e\\ Je 
    \end{array}\right)    
    =-(d A+A\wedge A)\wedge 
    \left(\begin{array}{c}
        0\\ e\\ Je 
    \end{array}\right)\\
    &=-F_A\wedge 
    \left(\begin{array}{c}
        0\\ e\\ Je 
    \end{array}\right)
    =-\left(\begin{array}{ccc} 
        0 & 0 & 0 \\ 
        0 & \alpha & \beta \\
        0 &-\beta & \alpha
    \end{array}\right)\wedge 
    \left(\begin{array}{c}
        0\\ e\\ Je 
    \end{array}\right)\\
    &=-\left(\begin{array}{c} 0\\ \alpha\wedge e+\beta\wedge Je\\ \alpha\wedge Je-\beta\wedge e\end{array}\right).
  \end{align*}
Equation \eqref{eq:F.wedge.e.Je.2} follows similarly from the structure equations \eqref{eq:str.eqs.2}.
\end{proof}
\subsection{The ``squashed'' Bismut and Hull connections 
on \texorpdfstring{$TK$}{TK}}

We now introduce an additional parameter to our connections which introduces a multiple of the flux $H_{\epsilon}$ as torsion.  This, in particular, leads us to the Bismut and Hull connections.

\subsubsection{Local connection matrices and torsion
}

We begin by identifying the flux $H_{\epsilon}$ with a locally defined matrix of 1-forms and a vector-valued 2-form as follows, so that we can define connections with torsion given by the flux.

\begin{prop}
\label{prop:H.torsion} 
In the local coframe of Definition \ref{dfn:coframe}, and using the notation from Definition \ref{dfn:box.vector}, let
\begin{equation}
\label{eq:C}
    C:=\left(\begin{array}{ccc}
        0 & Je^{\rm T} & -e^{\rm T}\\ -Je & -[e] & e_0I+[Je] \\
        e & -e_0I+[Je] & [e]
    \end{array}\right)
    =
    \left(\begin{array}{ccc}
        0 & Je^{\rm T} & -e^{\rm T}\\ -Je & -[e] & [Je] \\
        e & [Je] & [e]
    \end{array}\right)
    -e_0
    \cI.
\end{equation}
    Then we may raise an index on the $3$-form $H_{\epsilon}$ and view it as a vector-valued $2$-form, as follows:
\begin{equation}
    H_{\epsilon}=\frac{\epsilon}{2}\left(\begin{array}{ccc}0 & Je^{\rm T} & -e^{\rm T}\\ -Je & -[e] & e_0I+[Je] \\
    e & -e_0I+[Je] & [e]\end{array}\right)\wedge\left(\begin{array}{c} e_0\\ e\\ Je\end{array}\right)=\frac{\epsilon}{2}C\wedge \left(\begin{array}{c} e_0\\ e\\ Je\end{array}\right).
\end{equation}
\end{prop}

\begin{proof}
By Lemma \ref{lem:flux}, \eqref{eq:omega.coframe} and \eqref{eq:Re.Omega.coframe}, we have that
\begin{equation}
\begin{split}    
H_{\epsilon}&=-\epsilon^2\eta\wedge\omega+\epsilon\Re\Omega\\
    &=-\epsilon e_0\wedge (e_1\wedge Je_1+e_2\wedge Je_2+e_3\wedge Je_3)\\
    &\qquad+\epsilon(e_1\wedge e_2\wedge e_3-e_1\wedge Je_2\wedge Je_3
    -e_2\wedge Je_3\wedge Je_1-e_3\wedge Je_1\wedge Je_2).
\end{split}
\end{equation}
We raise an index, so that $H_{\epsilon}$ is a vector-valued 2-form, and use Lemma \ref{lem: computational rules} to deduce the claim:
\begin{align*}
    H_{\epsilon}&=\epsilon
    \left(\begin{array}{c}
        -e_1\wedge Je_1-e_2\wedge Je_2-e_3\wedge Je_3 \\ e_0\wedge Je_1+e_2\wedge e_3-Je_2\wedge Je_3\\
        e_0\wedge Je_2+e_3\wedge e_1-Je_3\wedge Je_1\\
        e_0\wedge Je_3+e_1\wedge e_2-Je_1\wedge Je_2\\
        -e_0\wedge e_1-e_2\wedge Je_3+e_3\wedge Je_2\\
        -e_0\wedge e_1-e_3\wedge Je_1+e_1\wedge Je_3\\
        -e_0\wedge e_1-e_1\wedge Je_2+e_2\wedge Je_1
    \end{array}\right)
    =\frac{\epsilon}{2}
    \left(\begin{array}{ccc}
        0 & Je^{\rm T} & -e^{\rm T}\\ 
        -Je & -[e] & e_0I+[Je] \\
        e & -e_0I+[Je] & [e]
    \end{array}\right)
    \wedge\left(\begin{array}{c} 
        e_0\\ e\\ Je\end{array}\right).
        \qedhere
\end{align*}
\end{proof}

\begin{cor}
\label{cor:conn.delta.k}
    In the terms of Definition \ref{dfn:theta.eps.k} and Proposition \ref{prop:H.torsion}, let  $\tau_{\epsilon}:=\epsilon C$; then each local matrix
    \begin{equation}
    \label{eq:theta.delta.eps.k}
        \theta^{\delta,k}_{\epsilon}
        =\theta_{\epsilon}^{k} +\frac{k\delta}{2}\tau_\epsilon
        =A+\frac{k\epsilon}{2}B +\frac{k\epsilon\delta}{2}C,
        \qforq
        k\neq 0 \text{ and } \delta\in\R,
    \end{equation}
    defines a connection on $TK$, with torsion 
    \begin{equation}\label{eq:H.k.delta.eps}
    H^{\delta,k}\ep=(1-k) \epsilon\omega\otimes e_0+k\delta H_\epsilon.
    \end{equation}
    Explicitly,
    \begin{equation*}
        \theta^{\delta,k}_{\epsilon}
        =A+\frac{k\epsilon}{2}B +\frac{k\epsilon\delta}{2}C
        =\left(\begin{array}{ccc} 
            0 & \frac{k\epsilon(1+\delta)}{2}Je^{\rm T} & -\frac{k\epsilon(1+\delta)}{2}e^{\rm T}\\[2pt]
            -\frac{k\epsilon(1+\delta)}{2}Je & a-\frac{k\epsilon\delta}{2}[e] & b-\frac{k\epsilon(1-\delta)}{2}e_0I+\frac{k\epsilon\delta}{2}[Je]\\[2pt]
            \frac{k\epsilon(1+\delta)}{2}e & -b+\frac{k\epsilon(1-\delta)}{2}e_0I+\frac{k\epsilon\delta}{2}[Je] & a+\frac{k\epsilon\delta}{2}[e]
        \end{array}\right).
    \end{equation*}
\end{cor}

\begin{proof} We see from \eqref{eq:str.eqs.matrix} and Proposition \ref{prop:H.torsion} that 
\begin{align*}
 d\left(\begin{array}{c} e_0\\ e\\ Je\end{array}\right)&=-(A+\frac{k\epsilon}{2}B+\frac{k\epsilon\delta}{2}C)\wedge    \left(\begin{array}{c} e_0\\ e\\ Je\end{array}\right)+(1-k)\epsilon\left(\begin{array}{c}\omega\\ 0\\0\end{array}\right)
 +\frac{k\epsilon\delta}{2}C\wedge \left(\begin{array}{c} e_0\\ e\\ Je\end{array}\right)r\\
 &=-
 \theta^{\delta,k}_{\epsilon}\wedge    
    \left(\begin{array}{c} 
        e_0\\ e\\ Je
    \end{array}\right)+(1-k)\epsilon\left(\begin{array}{c}\omega\\ 0\\0\end{array}\right)
    +k\delta H_{\epsilon}.\qedhere
\end{align*}
\end{proof}

\begin{remark}
It is possible to further deform the connection, and indeed the whole heterotic $\rG_2$ system, by allowing a non-trivial (non-constant) dilaton, which is equivalent to performing a conformal transformation on the $\rG_2$-structure.  However, since there are in general no distinguished functions on $K$ to define the dilaton, we will not pursue this possibility here. 
\end{remark}

\begin{definition}
\label{dfn:Bismut.k} 
    Taking $\delta=+1$ in \eqref{eq:theta.delta.eps.k} gives 
\begin{equation}
    \label{eq:Bismut.k}
    \theta^{+,k}_{\epsilon} =A+\frac{k\epsilon}{2}(B+C)
    = \left(
    \begin{array}{ccc} 
    0 & k\epsilon Je^{\rm T} & -k\epsilon e^{\rm T}\\
    -k\epsilon Je & a-\frac{k\epsilon}{2}[e] & b+\frac{k\epsilon}{2}[Je]\\
  k\epsilon e & -b+\frac{k\epsilon}{2}[Je] & a+\frac{k\epsilon}{2}[e]
    \end{array}
    \right).
\end{equation}
    We see from our choice of coframe that $\theta^{+,k}_{\epsilon}$ takes values in $\mathfrak{g}_2\subseteq\Lambda^2$, see e.g.~\cite{Lotay2011}, 
    and hence $\theta^{+,k}_{\epsilon}$ has holonomy contained in $\rG_2$, as its curvature will  necessarily take values in $\mathfrak{g}_2$.

Further, setting $k=1$ in \eqref{eq:Bismut.k} gives what is often called       
the \emph{Bismut connection} $\theta_{\epsilon}^+$ for $\varphi_{\epsilon}$, the unique metric connection which makes $\varphi_{\epsilon}$ parallel and has totally skew-symmetric torsion (which is the flux $H_{\epsilon}$).       
       \end{definition}
    
    \begin{remark}  The Bismut connection has been the subject of much study, and is a natural connection in this context. It is therefore tempting to use the Bismut connection (and more generally the connections $\theta^{+,k}_{\epsilon}$ in Definition \ref{dfn:Bismut.k}) when studying the heterotic $\rG_2$ system, particularly because of its holonomy property.   However, inspired by the ideas in \cites{Hull1986,Martelli2011}, one could also consider a connection, known as the \emph{Hull connection}, whose torsion has the \emph{opposite sign} to the Bismut connection when trying to satisfy the heterotic Bianchi identity  \eqref{eq:anomaly}.  This also motivates our discussion of the connections $\theta^{\delta,k}_{\epsilon}$ for $\delta<0$ as well as $\delta\geq 0$. 
    
    \end{remark}

    As a consequence of the previous remark, we will  also be \color{black} interested in the Hull connection, formally defined below.
\begin{definition}
\label{dfn:Hull.k}
    Taking $\delta=-1$ in \eqref{eq:theta.delta.eps.k} gives
\begin{align}
\label{eq:Hull.k}
\begin{split}
    \theta^{-,k}_\epsilon
    &= A+\frac{k\epsilon}{2}(B-C)\\
    &=
    \left(\begin{array}{ccc} 
        0 & 0 & 0\\
        0 & a+\frac{k\epsilon}{2}[e] & b-k\epsilon e_0I-\frac{k\epsilon}{2}[Je]\\
        0 & -b+ k\epsilon e_0I-\frac{k\epsilon}{2}[Je] & a-\frac{k\epsilon}{2}[e]
    \end{array}\right)\\
    &=
    \left(\begin{array}{ccc} 
        0 & 0 & 0\\
        0 & a+\frac{k\epsilon}{2}[e] & b -\frac{k\epsilon}{2}[Je]\\
        0 & -b-\frac{k\epsilon}{2}[Je] & a-\frac{k\epsilon}{2}[e]
    \end{array}\right)
    +k\epsilon e_0\cI. 
\end{split}
\end{align}
Setting $k=1$ in \eqref{eq:Hull.k} gives  the \emph{Hull connection} $\theta^{-}_{\epsilon}$ associated to the $\rG_2$-structure $\varphi_{\epsilon}$.
\end{definition}

\begin{remark}
As in the case of $\theta^k_{\epsilon}$, we may view the connections $\theta^{+,k}_{\epsilon}$ and $\theta^{-,k}_{\epsilon}$, respectively,  as ``squashed'' versions of the Bismut and Hull connections $\theta^+_{\epsilon}$ and $\theta^-_{\epsilon}$.
\end{remark}

\subsubsection{Local curvature matrices
}

Now, we want to determine 
the curvature of $\theta_{\epsilon}^{\delta,k}$ in Corollary \ref{cor:conn.delta.k}, with a particular emphasis on the cases $\delta=\pm 1$.  We begin with the result for all $\delta$. 

\begin{prop}
\label{prop:curv.k} 
    The curvature $R^{\delta,k}_\epsilon$ of the connection $\theta_{\epsilon}^{\delta,k}$ in \eqref{eq:theta.delta.eps.k} satisfies
\begin{equation}
\label{eq:R.delta.k}
    R^{\delta,k}_\epsilon
    =F_A 
    +\frac{k\epsilon^2(1-\delta)}{2}\omega\mathcal{I}
    + \frac{k^2\epsilon^2}{4}Q^{\delta}
    ,
\end{equation}
where $\mathcal{I}$ is given in \eqref{eq:mathcal.I}, 
\begin{align}
\label{eq:BdeltaC2}
\begin{split}
    Q^{\delta}:=(B+\delta C)&\wedge (B+\delta C)=(1-\delta)Q_-^{\delta}+(1+\delta)Q_+^{\delta}+\delta^2Q_0
\end{split}
\end{align}
and
\begin{align}
    Q^{\delta}_-&=e_0 \wedge \left(\begin{array}{ccc} 
        0 & (1+\delta) 
        e^{\rm T} & (1+\delta) 
        Je^{\rm T}\\ -(1+\delta) 
        e & -2\delta[
        Je] & -2\delta[
        e] \\
        -(1+\delta) 
        Je & -2\delta[
        e] & 2\delta[
        Je]
    \end{array}\right), \label{eq:Qdelta-}\\
 Q^{\delta}_+&=   \left(\begin{array}{ccc} 
        0 & 2\delta(e\times Je)^{\rm T} & \delta(e\times e-Je\times Je)^{\rm T}\\ -2\delta(e\times Je) & -(1+\delta)(Je\wedge Je^{\rm T}) & (1+\delta)(Je\wedge e^{\rm T})\\
        -\delta(e\times e-Je\times Je) & (1+\delta)(e\wedge Je^{\rm T}) & - (1+\delta)(e\wedge e^{\rm T})
    \end{array}\right),
    \label{eq:Qdelta+}\\
    Q_0&=\frac{1}{2} \left(\begin{array}{ccc}
        0 & 0 & 0 \\ 
        0 & -[e\times e+Je\times Je] & -2([e]\wedge [Je]-[Je]\wedge [e])\\
        0 & 2([e]\wedge [Je] -[Je]\wedge [e]) & -[e\times e+Je\times Je]
    \end{array}\right).
    \label{eq:Q0}
\end{align}
\end{prop}

\begin{proof}
We begin by observing that, by Corollary \ref{cor:conn.delta.k} and \eqref{eq:dAB},
\begin{align}
    R^{\delta,k}_\epsilon
    &=d\theta_{\epsilon}^{\delta,k} +\theta_{\epsilon}^{\delta,k} \wedge\theta_{\epsilon}^{\delta,k}
    \nonumber\\
    &=d A +\frac{k\epsilon}{2}dB +\frac{k\delta\epsilon}{2}d C+ \left(A+\frac{k\epsilon}{2}B +\frac{k\epsilon\delta}{2}C\right)\wedge \left(A+\frac{k\epsilon}{2}B +\frac{k\epsilon\delta}{2}C\right)
    \nonumber\\
    &= F_A+ \frac{k\epsilon}{2}(d B+ A\wedge B+B\wedge A) +\frac{k\epsilon\delta}{2}(dC +A\wedge C+C\wedge A)+ \frac{k^2\epsilon^2}{4}(B+\delta C)\wedge (B+\delta C)
    \nonumber\\
    &=F_A+\frac{k\epsilon^2}{2}\omega\mathcal{I}+\frac{k\epsilon\delta}{2}(dC +A\wedge C+C\wedge A)+ \frac{k^2\epsilon^2}{4}(B+\delta C)\wedge (B+\delta C).
    \label{eq:R.delta.k.1}
\end{align}

We may easily compute $dC+A\wedge C+C\wedge A$ appearing in \eqref{eq:R.delta.k.1}.
We first see that
\begin{equation*}
    (dC+A\wedge C+C\wedge A)_{1j}
    =(dB+A\wedge B+B\wedge A)_{1j}=0.
\end{equation*}
Therefore,
\begin{equation*}
    (dC+A\wedge C+C\wedge A)_{j1}=0
\end{equation*}
as well by skew-symmetry.  We may therefore write $dC+A\wedge C+C\wedge A$ in the block form
\begin{equation*}
dC+A\wedge C+C\wedge A=\left(\begin{array}{ccc}
    0 & 0 & 0  \\
    0 & c & d \\
    0 & -d^{\rm T} & -c
\end{array}\right).
\end{equation*}
We then find that
\begin{align*}
    c&=-b\wedge e_0I-e_0I\wedge b -d([e])-a\wedge [e]-[e]\wedge a+b\wedge [Je]-[Je]\wedge b=0
\end{align*}
using the structure equations \eqref{eq:str.eqs.2} in Proposition \ref{prop:str.eqs.2}.  We also find that
\begin{align*}
    d&=d(e_0)I+d([Je])+a\wedge e_0I+e_0I\wedge a+a\wedge [Je]+[Je]\wedge a+b\wedge [e]-[e]\wedge b\\
    &=\epsilon\omega I,
\end{align*}
using \eqref{eq:str.eqs.1} and  \eqref{eq:str.eqs.2}.
Overall, we deduce that
\begin{equation*}
dC+A\wedge C+C\wedge A=\epsilon\omega\left(\begin{array}{ccc} 0 & 0 & 0\\
0  & 0  & I  \\
0 & -I     & 0
\end{array}\right)=-\epsilon\omega\mathcal{I}.
\end{equation*}
Hence, \eqref{eq:R.delta.k} follows.

We now need only verify \eqref{eq:BdeltaC2}.  Recall from Corollary \ref{cor:conn.delta.k} that
\begin{equation}\label{eq:BdeltaC}
    B+\delta C=\left(\begin{array}{ccc}
0  & (1+\delta)Je^{\rm T}  & -(1+\delta)e^{\rm T} \\
-(1+\delta)Je & -\delta[e]    & -(1-\delta)e_0I+\delta[Je]\\
(1+\delta)e &  (1-\delta)e_0I+\delta[Je]  & \delta[e]
\end{array}\right).
\end{equation}
Using Lemma \ref{lem: computational rules}, we start with the first row of $(B+\delta C)\wedge (B+\delta C)$ and find the non-zero entries
\begin{align*}
    (1-\delta)(1+\delta)e_0\wedge e^{\rm T}-\delta(1+\delta)\big(Je^{\rm T}\wedge [e] &+e^{\rm T}\wedge [Je]\big)  =(1-\delta)(1+\delta)e_0\wedge e^{\rm T}+2\delta(1+\delta)(e\times Je)^{\rm T}
\end{align*}
and 
\begin{align*}
 (1-\delta)(1+\delta)e_0\wedge Je^{\rm T}&+   \delta(1+\delta)\big(Je^{\rm T}\wedge [Je]-e^{\rm T}\wedge [e]\big)\\
&    =(1-\delta)(1+\delta)e_0\wedge Je^{\rm T}+\delta(1+\delta)(e\times e-Je\times Je)^{\rm T}.    
\end{align*}
Moving to the middle block and again using Lemma \ref{lem: computational rules}, we obtain
\begin{align*}
    -(1+\delta)^2Je\wedge Je^{\rm T}
    &+\delta^2\big([e]\wedge [e]+[Je]\wedge [Je]\big)-\delta(1-\delta)(e_0I\wedge [Je]-[Je]\wedge e_0I) \\
    &=-(1+\delta)^2Je\wedge Je^{\rm T} -\frac{1}{2}\delta^2[e\times e +Je\times Je] -2\delta(1-\delta)e_0\wedge[ Je].
\end{align*}
Similarly, for the bottom right block, we obtain
\begin{align*}
    -(1+\delta)^2e\wedge e^{\rm T}-\frac{1}{2}\delta^2[e\times e+Je\times Je]+2\delta(1-\delta)e_0\wedge[ Je].
\end{align*}
The remaining entries are defined by the middle right block, which is
\begin{align*}
    (1+\delta)^2 Je\wedge e^{\rm T}-\delta^2([e]\wedge [Je]-[Je]\wedge[e])-2\delta(1-\delta)e_0\wedge [e].
\end{align*}
Equation \eqref{eq:BdeltaC2} now follows.
\end{proof}

We now can specialize to the Bismut and Hull connections.

\begin{cor}
The curvature $R_{\theta^{+}_\epsilon}$ of the Bismut connection $\theta^{+}_\epsilon$ satisfies
\begin{equation}\label{eq:R.+}
R_{\theta^{+}_\epsilon}=F_A+ \frac{\epsilon^2}{4}(B+C)\wedge (B+C),
\end{equation}
where 
\begin{align}
\label{eq:BdeltaC2.+}
\begin{split}
    (B+C)\wedge (B+ C)
    =\;&2
    \left(\begin{array}{ccc} 
        0 & 2(e\times Je)^{\rm T} & (e\times e-Je\times Je)^{\rm T}\\ -2(e\times Je) & -2(Je\wedge Je^{\rm T}) & 2(Je\wedge e^{\rm T})\\
        -(e\times e-Je\times Je) & 2(e\wedge Je^{\rm T}) & - 2(e\wedge e^{\rm T})
    \end{array}\right)\\ 
    &+\frac{1}{2}
    \left(\begin{array}{ccc} 
        0 & 0 & 0 \\ 
        0 & -[e\times e+Je\times Je] & -2([e]\wedge [Je]-[Je]\wedge [e])\\
        0 & 2([e]\wedge [Je] -[Je]\wedge [e]) & -[e\times e+Je\times Je]
    \end{array}\right).
\end{split}
\end{align}
\end{cor}

\begin{cor}
The curvature $R_{\theta_{\epsilon}^{-}}$ of the Hull connection $\theta_{\epsilon}^{-}$ satisfies
\begin{equation}
\label{eq:R.-}
    R_{\theta^{-}_\epsilon}
    =F_A +\epsilon^2\omega\mathcal{I}+ \frac{\epsilon^2}{4}(B- C)\wedge (B- C),
\end{equation}
where $\mathcal{I}$ is given in \eqref{eq:mathcal.I} and
\begin{align}
\label{eq:BdeltaC2.-}
\begin{split}
    (B- C)&\wedge (B- C)\\
    =\;& 4e_0\wedge 
    \left(\begin{array}{ccc} 
        0 & 0 & 0\\ 
        0 & [Je] & [e] \\
        0 & [e] & -[Je]
    \end{array}\right)
    +\frac{1}{2}
    \left(\begin{array}{ccc} 
        0 & 0 & 0 \\ 
        0 & -[e\times e+Je\times Je] & -2([e]\wedge [Je]-[Je]\wedge [e])\\
        0 & 2([e]\wedge [Je] -[Je]\wedge [e]) & -[e\times e+Je\times Je]
    \end{array}\right).
\end{split}
\end{align}
\end{cor}

\subsection{Connections: an extra twist}

It will be useful to ``twist'' our connection by multiples of $e_0\mathcal{I}$.  To discern the impact of this twist on the curvature of the connection, we have the following lemma.

\begin{lemma}
\label{lem:twist}
    The local connection matrices $A,B,C$ from Corollary \ref{cor:conn.eps} and Proposition \ref{prop:H.torsion} satisfy
\begin{align}
    A\wedge e_0\mathcal{I}+e_0\mathcal{I}\wedge A&=0,
\label{eq:twist.A}\\
    B\wedge e_0\mathcal{I}+e_0\mathcal{I}\wedge  B &=e_0\wedge 
    \left(\begin{array}{ccc} 
        0 & e^{\rm T} & Je^{\rm T}\\
        -e & 0 & 0\\
        -Je & 0 & 0
    \end{array}\right),
\label{eq:twist.B}\\
    C\wedge e_0\mathcal{I}+e_0\mathcal{I}\wedge C &=e_0\wedge 
    \left(\begin{array}{ccc} 
        0 & e^{\rm T} & Je^{\rm T}\\
        -e & -2[Je] & -2[e]\\
        -Je & -2[e] & 2[Je]
    \end{array}\right).
\label{eq:twist.C}
\end{align}
\end{lemma}

\begin{proof}
Given that $A$ in \eqref{eq:A} takes values in $\mathfrak{su}(3)\subseteq\mathfrak{u}(3)$ and $\mathcal{I}$ in \eqref{eq:mathcal.I} is central in $\mathfrak{u}(3)$, we immediately deduce \eqref{eq:twist.A}.
Moreover, we see from \eqref{eq:B}, \eqref{eq:C} and \eqref{eq:mathcal.I} that
$$B\wedge e_0\mathcal{I}=\left(\begin{array}{ccc} 0 & -(e\wedge e_0)^{\rm T} & -(Je\wedge e_0)^{\rm T}\\
0 & 0 & 0\\
0 & 0 & 0
\end{array}\right),\qquad e_0\mathcal{I}\wedge B=\left(\begin{array}{ccc} 0 & 0 & 0 \\
-e_0\wedge e & 0 & 0\\
-e_0\wedge Je & 0 & 0\end{array}\right)$$
and
$$C\wedge e_0\mathcal{I}=\left(\begin{array}{ccc} 0 & -(e\wedge e_0)^{\rm T} & -(Je\wedge e_0)^{\rm T}\\
0 & [Je]\wedge e_0 & [e]\wedge e_0\\
0 & [e]\wedge e_0  & -[Je]\wedge e_0 \end{array}\right),$$ $$e_0\mathcal{I}\wedge C=\left(\begin{array}{ccc}0 & 0 & 0\\
-e_0\wedge e & -e_0\wedge [Je] & -e_0\wedge [e]\\
-e_0\wedge Je & -e_0\wedge [e] & e_0\wedge [Je]\end{array}\right).$$
Equations \eqref{eq:twist.B} and \eqref{eq:twist.C} then follow.
\end{proof}

The previous lemma allows us to compute the curvature of a twisted connection, in particular establishing Theorem \ref{thm: Main Theorem}--(iii),  as follows.

\begin{prop}
\label{prop:twist.curvature}
    In the local coframe from Definition \ref{dfn:coframe}, define a connection $\theta^{\delta,k}_{\epsilon,m} $ on $TK$ by
    \begin{equation}
    \label{eq:theta.delta.m.k}
    \theta^{\delta,k}_{\epsilon,m} =\theta^{\delta,k}_{\epsilon}+\frac{km\epsilon}{2}e_0\mathcal{I}.
    \end{equation}
    Then its torsion is
    \begin{equation}\label{eq:H.k.delta.eps.m}
    H^{\delta,k}_{\epsilon,m} =\big(1-k-\frac{km}{2}\big)\epsilon\omega\otimes e_0+\frac{km\epsilon}{2}e_0\wedge\omega+k\delta H_{\epsilon}    
    \end{equation}
    and its curvature is given by
    \begin{equation}
    \label{eq:curv.delta.m.k}
        R^{\delta,k}_{\epsilon,m} 
        =F_A  +\frac{k\epsilon^2(1-\delta+m)}{2}\omega\mathcal{I}  +\frac{k^2\epsilon^2}{4}Q^{\delta}_m
    \end{equation}
    where 
    \begin{equation}
    \label{eq:Q.delta.m}
    Q^{\delta}_m=
    \left(1-\delta+m\right)Q_-^{\delta} +(1+\delta)Q_+^{\delta}+\delta^2Q_0
    \end{equation}
    for $Q^{\delta}_-,Q^{\delta}_+,Q_0$ defined in \eqref{eq:Qdelta-},  \eqref{eq:Qdelta+} and \eqref{eq:Q0}, respectively. 
\end{prop}

\begin{proof}
Using \eqref{eq:H.k.delta.eps} we see that
\begin{align*}
    d\left(\begin{array}{c} e_0\\ e\\ Je\end{array}\right)&=-\theta^k_\delta\wedge\left(\begin{array}{c} e_0\\ e\\ Je\end{array}\right)+(1-k)\epsilon\left(\begin{array}{c} \omega \\ 0\\ 0 \end{array}\right)+k\delta H_{\epsilon}\\
    &=-\theta^{\delta,k}_{\epsilon,m} \wedge\left(\begin{array}{c} e_0\\ e\\ Je\end{array}\right)+\frac{km\epsilon}{2}e_0\mathcal{I}\wedge\left(\begin{array}{c} e_0\\ e\\ Je\end{array}\right)+(1-k)\epsilon\left(\begin{array}{c} \omega \\ 0\\ 0 \end{array}\right)+k\delta H_{\epsilon}.
\end{align*}
Since 
\begin{equation*}
    \frac{km\epsilon}{2}e_0\mathcal{I}\wedge\left(\begin{array}{c} e_0\\ e\\ Je\end{array}\right)=\frac{km\epsilon}{2}\left(\begin{array}{c}0\\ -e_0\wedge Je\\ e_0\wedge e\end{array}\right)
\end{equation*}
and raising an index on $e_0\wedge\omega$ gives the vector-valued 2-form
$$\left(\begin{array}{c} \omega \\ -e_0\wedge Je\\ e_0\wedge e \end{array}\right),$$
we quickly deduce \eqref{eq:H.k.delta.eps.m}.

We know by definition that
\begin{align*}
    R^{\delta,k}_{\epsilon,m} &=d(\theta^{\delta,k}_{\epsilon}+\frac{km\epsilon}{2}e_0\mathcal{I})+(\theta^{\delta,k}_{\epsilon}+\frac{km\epsilon}{2}e_0\mathcal{I})\wedge (\theta^{\delta,k}_{\epsilon}+\frac{km\epsilon}{2}e_0 \mathcal{I})\\
    &=
R^{\delta,k}_\epsilon+\frac{km\epsilon^2}{2}\omega\mathcal{I}+\frac{km\epsilon}{2}(\theta^{\delta,k}_{\epsilon}\wedge e_0 \mathcal{I}+e_0 \mathcal{I}\wedge\theta^{\delta,k}_{\epsilon}).
\end{align*}
Lemma \ref{lem:twist} implies that
\begin{align*}
    \theta^{\delta,k}_{\epsilon}\wedge e_0\mathcal{I}+e_0\mathcal{I}\wedge\theta^{\delta,k}_{\epsilon}&=\left(A+\frac{k\epsilon}{2}(B+\delta C)\right)\wedge e_0\mathcal{I}+e_0\mathcal{I}\wedge\left(A+\frac{k\epsilon}{2}(B+\delta C)\right)\\
    &=\frac{k\epsilon}{2} e_0\wedge \left(\begin{array}{ccc}
    0 & (1+\delta)e^{\rm T} & (1+\delta) Je^{\rm T}\\
    -(1+\delta)e & -2\delta [Je] & -2\delta [e]\\
    -(1+\delta)Je & -2\delta[e] & 2\delta[Je]
    \end{array}\right)=\frac{k\epsilon}{2}Q^{\delta}_-
\end{align*}
by \eqref{eq:Qdelta-}.
  The result now follows from Proposition \ref{prop:curv.k}.
\end{proof}

The following observation, which may have potential interest, is immediate from \eqref{eq:H.k.delta.eps.m}:

\begin{cor}
    The connection $\theta^{\delta,k}_{\epsilon,m} $ in \eqref{eq:theta.delta.m.k} has totally skew-symmetric torsion if, and only if,
    $$
    1-k\left(1+\frac{m}{2}\right)=0.
    $$
\end{cor}

\subsection{The \texorpdfstring{G\textsubscript{2}}{G2}-instanton condition}

One way to check the $\rG_2$-instanton condition is to verify the vanishing of the wedge product of the curvature with $\psi_\epsilon$, cf.~\eqref{eq:psi.eps}. Before doing this, we make some elementary observations.

\begin{lemma}
\label{lem:ImOmega.identities} 
    In the local coframe \eqref{eq:coframe} on a contact Calabi--Yau $7$-manifold as in Definition \ref{dfn:G2.str}, and using the notation from Definition \ref{dfn:box.vector}, the following identities hold:
\begin{align}
    2(e\times Je)\wedge\Im\Omega &= 4e\wedge\frac{\omega^2}{2},&
    (e\times e-Je\times Je)\wedge\Im\Omega &= 4Je\wedge \frac{\omega^2}{2},\label{eq:exJe.ImOmega}\\
    e\wedge e^{\rm T}\wedge\Im\Omega &=[Je]\wedge \frac{\omega^2}{2}, &
    Je\wedge Je^{\rm T}\wedge\Im\Omega &=-[Je]\wedge \frac{\omega^2}{2},\label{eq:ee.ImOmega}\\
    [e\times e+Je\times Je]\wedge\Im\Omega&=0, & ([e]\wedge [Je]-[Je]\wedge[e])\wedge\Im\Omega &=0,\label{eq:eJe.ImOmega}\\
    [e\times e+Je\times Je]\wedge\frac{\omega^2}{2}&=0, & 
    ([e]\wedge[Je]-[Je]\wedge [e])\wedge\frac{\omega^2}{2}&=-4\frac{\omega^3}{6}I,\label{eq:eJe.omega1}\\
    e\wedge Je^{\rm T}\wedge\Im\Omega &=[e]\wedge\frac{\omega^2}{2}, & e\wedge Je^{\rm T}\wedge\frac{\omega^2}{2}&=\frac{\omega^3}{6}I.\label{eq:eJe.omega2}\\
    Je\wedge e^{\rm T}\wedge\Im\Omega &=[e]\wedge\frac{\omega^2}{2}, & Je\wedge e^{\rm T}\wedge\frac{\omega^2}{2}&=-\frac{\omega^3}{6}I.\label{eq:eJe.omega3}
\end{align}
\end{lemma}

\begin{proof}
We observe from \eqref{eq:Im.Omega.coframe} that
\begin{align*}
    \Im\Omega\wedge e_2\wedge Je_3
    =Je_2\wedge e_3\wedge e_1\wedge e_2\wedge Je_3
    =e_1\wedge (e_2\wedge Je_2\wedge e_3\wedge Je_3)=e_1\wedge \frac{\omega^2}{2}
    \end{align*}
    and
\begin{align*}
    \Im\Omega\wedge e_3\wedge Je_2    =Je_3\wedge e_1\wedge e_2\wedge e_3\wedge Je_2=- =e_1\wedge (e_2\wedge Je_2\wedge e_3\wedge Je_3)
    =-e_1\wedge \frac{\omega^2}{2}.
    \end{align*}   
Similarly, we may also compute
\begin{align*}
\Im\Omega\wedge e_2\wedge e_3 &=
-Je_1\wedge Je_2\wedge Je_3\wedge e_2\wedge e_3= ( e_2\wedge Je_2\wedge e_3\wedge Je_3)\wedge Je_1 
=\frac{\omega^2}{2}\wedge Je_1.
\end{align*}
and 
\begin{align*}
    \Im\Omega\wedge Je_2\wedge Je_3=Je_1\wedge e_2\wedge e_3\wedge Je_2\wedge Je_3=-( e_2\wedge Je_2\wedge e_3\wedge Je_3)\wedge Je_1  =-\frac{\omega^2}{2}\wedge Je_1.
\end{align*}
Hence, \eqref{eq:exJe.ImOmega}, \eqref{eq:ee.ImOmega} and the first equations in \eqref{eq:eJe.omega2} and \eqref{eq:eJe.omega3} hold (noting that $e_j\wedge Je_j\wedge \Im\Omega=0$).  


We also notice that 
\begin{align*}
    e_1\wedge Je_1\wedge\frac{\omega^2}{2}&=e_1\wedge Je_1\wedge e_2\wedge Je_2\wedge e_3\wedge Je_3 =\frac{\omega^3}{6},
\end{align*}
from which the remaining identities in \eqref{eq:eJe.omega2} and \eqref{eq:eJe.omega3} follow (since clearly $e_j\wedge Je_k\wedge\omega^2=0$ for $j\neq k$).

The previous calculation, together with Lemma \ref{lem:basic} and \eqref{eq:diag}, show that
\begin{align*}
([e]\wedge [Je]-[Je]\wedge [e])\wedge\frac{\omega^2}{2}=-4 (e_1\wedge Je_1\wedge e_2\wedge Je_2\wedge e_3\wedge Je_3) I=-4\frac{\omega^3}{6}I
\end{align*}
as claimed.  The rest of 
\eqref{eq:eJe.omega1} follows from Lemma \ref{lem:basic}.
\end{proof}

\begin{prop}
\label{prop:inst.delta.k}
    The curvature $R_{\theta_{\epsilon}^{\delta,k}}$ of the connection $\theta_{\epsilon}^{\delta,k}$ in \eqref{eq:theta.delta.eps.k} satisfies
\begin{equation}
    \begin{split}\label{eq:G2.inst.delta.eps.k}
      R_{\theta_{\epsilon}^{\delta,k}} \wedge\psi_{\epsilon}
    &=
    \frac{k\epsilon^2(1-\delta)\big(6+k(1+3\delta)\big)}{4}\frac{\omega^3}{6}\mathcal{I} \\
    &+\frac{k^2\epsilon^2
    }{4} e_0\wedge\frac{\omega^2}{2} \wedge \left(\begin{array}{ccc}
       0 & (1-5\delta)(1+\delta)e^{\rm T} & (1-5\delta)(1+\delta)Je^{\rm T}\\
        (5\delta -1)(1+\delta)e & 
        (\delta^2-4\delta-1)[Je] & 
        (\delta^2-4\delta-1)[e]\\
        (5\delta-1)(1+\delta)Je & 
        (\delta^2-4\delta-1)[e] & 
        -(\delta^2-4\delta-1)[Je]
    \end{array}\right).      
    \end{split}
\end{equation}
Therefore, $\theta^{\delta,k}_{\epsilon}$ is never a $\rG_2$-instanton.
\end{prop}

\begin{remark}
We see that $\theta^{\delta,k}_{\epsilon}$ can  be a $\rG_2$-instanton if and only if we are in the trivial case where $k=0$, which we have excluded.
\end{remark}
\begin{proof}
Since $A$ is a $\rG_2$-instanton by Lemma \ref{lem:A}, we deduce immediately from Proposition \ref{prop:curv.k} that
\begin{align}
    R_{\theta_{\epsilon}^{\delta,k}}\wedge \psi_{\epsilon}
    &= F_A\wedge\psi_{\epsilon} +\frac{k\epsilon^2(1-\delta)}{2}(\omega \wedge\psi_{\epsilon})\mathcal{I} +\frac{k^2\epsilon^2}{4}Q^{\delta} 
    \wedge\psi_{\epsilon}\nonumber\\
    &=\frac{k\epsilon^2(1-\delta)}{4}\omega^3\mathcal{I}+\frac{k^2\epsilon^2}{4}Q^{\delta}
    \wedge \psi_{\epsilon}. 
    \label{eq:G2.inst.delta.eps.k.1}
\end{align}

We now study the term $Q^{\delta}
\wedge\psi_{\epsilon}$.  We first note that
\begin{align*}
    e_0\wedge e\wedge\psi_{\epsilon}
    &=e_0\wedge \frac{\omega^2}{2}\wedge e,\qquad
    e_0\wedge Je\wedge\psi_{\epsilon}
    =e_0\wedge \frac{\omega^2}{2}\wedge Je.
    \end{align*}
Hence, from \eqref{eq:Qdelta-}, we find that
\begin{equation}\label{eq:Qdelta-.inst}
Q^{\delta}_-\wedge\psi_{\epsilon}=e_0 \wedge\frac{1}{2}\omega^2\wedge \left(\begin{array}{ccc} 
        0 & (1+\delta) 
        e^{\rm T} & (1+\delta) 
        Je^{\rm T}\\ -(1+\delta) 
        e & -2\delta[
        Je] & -2\delta[
        e] \\
        -(1+\delta) 
        Je & -2\delta[
        e] & 2\delta[
        Je]
    \end{array}\right).    
\end{equation}

By Lemmas \ref{lem:basic} and \ref{lem:ImOmega.identities} we find that
\begin{align*}
    2(e\times Je)\wedge\psi_{\epsilon}
    &=-2e_0\wedge \Im\Omega\wedge(e\times Je)=-4e_0\wedge\frac{\omega^2}{2}\wedge e,\\
    (e\times e-Je\times Je)\wedge\psi_{\epsilon}
    &=-e_0\wedge \Im\Omega \wedge(e\times e-Je\times Je)=-4e_0\wedge\frac{\omega^2}{2}\wedge Je.
\end{align*}
%
We also see from Lemma \ref{lem:ImOmega.identities} that
\begin{align*}
    Je\wedge Je^{\rm T}\wedge \psi_{\epsilon}
    &=-e_0 \wedge \Im\Omega \wedge Je\wedge Je^{\rm T}=e_0 \wedge \frac{\omega^2}{2}\wedge [Je],\\
    e\wedge e^{\rm T}\wedge \psi_{\epsilon}
    &=-e_0 \wedge \Im\Omega \wedge e\wedge e^{\rm T}=-e_0 \wedge\frac{\omega^2}{2}\wedge [Je],\\
    Je\wedge e^{\rm T}\wedge \psi_{\epsilon}&=-\frac{\omega^3}{6}I-e_0\wedge \Im\Omega\wedge Je\wedge e^{\rm T}=-\frac{\omega^3}{6}I-e_0\wedge\frac{\omega^2}{2}\wedge [e],\\
      e\wedge Je^{\rm T}\wedge \psi_{\epsilon}&=\frac{\omega^3}{6}I-e_0\wedge \Im\Omega\wedge e\wedge Je^{\rm T}=\frac{\omega^3}{6}I-e_0\wedge\frac{\omega^2}{2}\wedge [e].
\end{align*}
We deduce that
\begin{align*}
Q_+^{\delta}\wedge\psi_{\epsilon}=(1+\delta)\frac{\omega^3}{6}\mathcal{I}+e_0\wedge\frac{\omega^2}{2}\wedge \left(\begin{array}{ccc} 0 & -4\delta e^{\rm T} & -4\delta Je^{\rm T}\\
4\delta e & -(1+\delta)[Je] & -(1+\delta)[e] \\
4\delta Je &-(1+\delta)[e] & (1+\delta)[Je] \end{array}\right).  
\end{align*}

Finally, it follows from Lemma \ref{lem:ImOmega.identities}
  that
  \begin{align*}
    [e\times e+Je\times Je]\wedge\psi_{\epsilon}=0,\qquad ([e]\wedge [Je]-[Je]\wedge [e])\wedge\psi_{\epsilon}=-4\frac{\omega^3}{6}I.  
  \end{align*}
 Thus, 
\begin{align*}
Q_0\wedge\psi_{\epsilon}= \frac{\omega^3}{6}\left(\begin{array}{ccc} 0 & 0 & 0\\
0 & 0 & 4I\\
0 & -4I & 0
\end{array}\right)=-4\frac{\omega^3}{6}\mathcal{I}.
\end{align*}

Overall, we have 
\begin{align*}
    Q^{\delta}&\wedge\psi_{\epsilon}=\big((1-\delta)Q_-^{\delta}+(1+\delta)Q_+^{\delta}+\delta^2Q_0\big)\wedge\psi_\epsilon\\
    &=(1-\delta)e_0 \wedge\frac{1}{2}\omega^2\wedge \left(\begin{array}{ccc} 
        0 & (1+\delta) 
        e^{\rm T} & (1+\delta) 
        Je^{\rm T}\\ -(1+\delta) 
        e & -2\delta[
        Je] & -2\delta[
        e] \\
        -(1+\delta) 
        Je & -2\delta[
        e] & 2\delta[
        Je]
    \end{array}\right)\\
    &\qquad+(1+\delta)^2\frac{\omega^3}{6}\mathcal{I}+(1+\delta)e_0^{(k)}\wedge\frac{\omega^2}{2}
    \wedge 
    \left(\begin{array}{ccc} 
        0 & -4\delta e^{\rm T} & -4\delta Je^{\rm T}\\
        4\delta e & -(1+\delta)[Je] & -(1+\delta)[e] \\
        4\delta Je &-(1+\delta)[e] & (1+\delta)[Je] 
    \end{array}\right)\\
    &\qquad
    -4\delta^2\frac{\omega^3}{6}\mathcal{I}\\
    &=(1-\delta)(1+3\delta)\frac{\omega^3}{6}\mathcal{I} +e_0\wedge\frac{\omega^2}{2} \wedge 
    \left(\begin{array}{ccc} 
        0 & (1+\delta)(1-5\delta)e^{\rm T} & (1+\delta)(1-5\delta)Je^{\rm T} \\
        (1+\delta)(5\delta-1)e & (\delta^2-4\delta-1)[Je] & (\delta^2-4\delta-1)[e] \\
        (1+\delta)(5\delta-1)Je & (\delta^2-4\delta-1)[e] & -(\delta^2-4\delta-1)[Je]
    \end{array}\right).
\end{align*}
We deduce from this equation and \eqref{eq:G2.inst.delta.eps.k.1} that the coefficient of $\frac{\omega^3}{6}\mathcal{I}$ in $R_{\theta_{\epsilon}^{\delta,k}}\wedge\psi_{\epsilon}$ is
\begin{equation*}
    \frac{6k\epsilon^2(1-\delta)}{4}+\frac{k^2\epsilon^2(1-\delta)(1+3\delta)}{4}=\frac{k\epsilon^2(1-\delta)\big(6+k(1+3\delta)\big)}{4}
\end{equation*}
The claimed formula \eqref{eq:G2.inst.delta.eps.k} now follows.

Since the quadratics $(1-5\delta)(1+\delta)$ and $\delta^2-4\delta-1$ in $\delta$ have no common roots, we see that if $\theta^{\delta,k}_{\epsilon}$ were a $\rG_2$-instanton, then we must have $k=0$.  
\end{proof}

\begin{remark}
In particular, we see that neither the Bismut nor the Hull connection are $\rG_2$-instantons. 
\end{remark}

A straightforward adaptation of the arguments leading to Proposition \ref{prop:inst.delta.k}, using Proposition \ref{prop:twist.curvature}, gives the following result for $\theta^{\delta,k}_{\epsilon,m} $.  

\begin{cor}\label{cor:inst.delta.m.k}
The curvature $R^{\delta,k}\epm$ of the connection $\theta^{\delta,k}_{\epsilon,m} $ in \eqref{eq:theta.delta.m.k} satisfies
\begin{equation}
\begin{split}
\label{eq:G2.inst.eps.delta.m.k}
    & R^{\delta,k}\epm \wedge\psi_{\epsilon}\\
    &=\frac{k\epsilon^2\big(6(1-\delta+m)+k(1-\delta)(1+3\delta)\big)}{4} \frac{\omega^3}{6}\mathcal{I} \\
    &+\frac{k^2\epsilon^2
    }{4} e_0\wedge\frac{\omega^2}{2} \wedge \left(\begin{array}{ccc}
       0 & (1+m-5\delta)(1+\delta)e^{\rm T} & (1+m-5\delta)(1+\delta)Je^{\rm T}\\
        (5\delta -1-m)(1+\delta)e & 
        (\delta^2-2(2+m)\delta-1)[Je] & 
        (\delta^2-2(2+m)\delta-1)[e]\\
        (5\delta-1-m)(1+\delta)Je & 
        (\delta^2-2(2+m)\delta-1)[e] & 
        -(\delta^2-2(2+m)\delta-1)[Je]
    \end{array}\right).
\end{split}
\end{equation}

Therefore, $\theta^{\delta,k}_{\epsilon,m} $ is never a $\rG_2$-instanton. 
\end{cor}

\begin{proof}
The key observation is \eqref{eq:Qdelta-.inst} which shows, together with Proposition \ref{prop:twist.curvature}, that we must add
\begin{equation*}
  \frac{km\epsilon^2}{4}\omega^3\mathcal{I}+  \frac{k^2\epsilon^2}{4}me_0 \wedge\frac{1}{2}\omega^2\wedge \left(\begin{array}{ccc} 
        0 & (1+\delta) 
        e^{\rm T} & (1+\delta) 
        Je^{\rm T}\\ -(1+\delta) 
        e & -2\delta[
        Je] & -2\delta[
        e] \\
        -(1+\delta) 
        Je & -2\delta[
        e] & 2\delta[
        Je]
    \end{array}\right)   
\end{equation*}
to the right-hand side of \eqref{eq:G2.inst.delta.eps.k} to obtain $  R_{\theta_{\epsilon,m}^{\delta,k}} \wedge\psi_{\epsilon}$.  The claimed formula \eqref{eq:G2.inst.eps.delta.m.k} then follows.

We deduce that, since $k\neq 0$,  $\theta^{\delta,k}_{\epsilon,m} $ is a $\rG_2$-instanton if and only if
\begin{align*}
    (1-\delta)(6+k(1+3\delta))+6m=0,\quad (5\delta -1-m)(1+\delta)=0,\quad (\delta^2-1)-2(2+m)\delta=0.
\end{align*}
One may see that the only real solutions have $\delta=-1$, meaning the second equation is satisfied for any $m$.  The third equation forces $m=-2$ and the first equation gives $12-4k+6m=0$, which then forces $k=0$.  
\end{proof}

\begin{remark}
\label{rem: approximate G2-instanton}
    Although $\theta^{\delta,k}_{\epsilon,m} $ is never a $\rG_2$-instanton, we \textcolor{black}{know} by  \eqref{eq:order.alpha} and \color{black} \eqref{eq:G2.inst.eps.delta.m.k} that it is an ``approximate'' $\rG_2$-instanton whenever
$$ \frac{k\epsilon^2\big(6(1-\delta+m)+k(1-\delta)(1+3\delta)\big)}{4},\quad \frac{k^2\epsilon^2
    }{4}(1+m-5\delta)(1+\delta), \quad \frac{k^2\epsilon^2
    }{4} (\delta^2-2(2+m)\delta-1)   $$
     are all ``sufficiently small'' in a suitable sense.  This smallness will be related to the constant $\alpha'$ which we will determine  in the next section on the anomaly-free condition \eqref{eq:anomaly}.
\end{remark}

\section{The anomaly term
}


We wish to study the heterotic Bianchi identity for the connections $\theta=\theta^{\delta,k}_{\epsilon,m} $ and $\rG_2$-structure $\varphi_{\epsilon}$.  By \eqref{eq:anomaly} and Lemma \ref{lem:flux}, this becomes
\begin{equation}
\label{eq:anomaly.eps.delta.m.k}
    dH_{\epsilon}=-\epsilon^2\omega^2=\frac{\alpha'}{4}(\tr F_A^2-\tr R_\theta^2).
\end{equation}
Proposition \ref{prop:twist.curvature} allows us to study when this condition can be satisfied, since by \eqref{eq:curv.delta.m.k}, we have that
\begin{equation}
\label{eq:anomaly.1}
\begin{split}
    R_\theta^2-F_A^2
    &=\frac{k^2\epsilon^4(1-\delta+m)^2}{4} \omega^2\mathcal{I}^2 +\frac{k\epsilon^2(1-\delta+m)}{2}(F_A\wedge\omega\mathcal{I} +\omega\mathcal{I}\wedge F_A)\\
    &\quad+\frac{k^3\epsilon^4(1-\delta+m)}{8}(\omega\mathcal{I}\wedge Q^{\delta}_m+Q^{\delta}_m\wedge\omega\mathcal{I}) +\frac{k^2\epsilon^2}{4}(F_A\wedge Q^{\delta}_m+Q^{\delta}_m\wedge F_A) +\frac{k^4\epsilon^4}{16}(Q^{\delta}_m)^2.
\end{split}
\end{equation}

\subsection{Terms involving the matrix \texorpdfstring{$\mathcal{I}$}{I}}

We begin by studying the trace of the first line on the right-hand side of \eqref{eq:anomaly.1}.  

\begin{lemma}
\label{lem:trace.1}
    For $\mathcal{I}$ as in \eqref{eq:mathcal.I} and $F_A$ as in \eqref{eq:FA} we have that
\begin{equation}
\label{eq:trace.I.1}
    \tr\mathcal{I}^2=-6 \qandq \tr(F_A\wedge\omega\mathcal{I}+\omega\mathcal{I}\wedge F_A)=0.
\end{equation}
\end{lemma}

\begin{proof}
We first notice that
\begin{equation*}
    \mathcal{I}^2
    =-\left(\begin{array}{ccc} 
        0 & 0 & 0 \\
        0 & I & 0\\ 
        0 & 0 & I
    \end{array}\right)
\end{equation*}
and hence the first equation in \eqref{eq:trace.I.1} holds.

We then deduce from the formula \eqref{eq:FA} for $F_A$ 
that
\begin{equation*}
F_A\wedge\omega\mathcal{I}+\omega\mathcal{I}\wedge F_A
=\left(\begin{array}{ccc} 0 & 0 & 0\\
0 & 2\beta\wedge\omega & -2\alpha\wedge\omega\\
0 & 2\alpha\wedge\omega & 2\beta\wedge\omega\end{array}\right).
\end{equation*}
Since $\beta$ is traceless, the second equation in \eqref{eq:trace.I.1} also holds.
\end{proof}
We deduce from \eqref{eq:anomaly.eps.delta.m.k} and Lemma \ref{lem:trace.1} that
\begin{equation}
\label{eq:trace.2}
\begin{split}
    \tr(R^2_{\theta^{\delta,k}_{\epsilon,m}  }-F_A^2)
    &=-\frac{3k^2\epsilon^4(1-\delta+m)^2}{2}\omega^2 +\frac{k^3\epsilon^4(1-\delta+m)}{8}\tr(\omega\mathcal{I}\wedge Q^{\delta}_m+Q^{\delta}_m\wedge\omega\mathcal{I})\\
    &\quad+\frac{k^2\epsilon^2}{4}\tr(F_A\wedge Q^{\delta}_m+Q^{\delta}_m\wedge F_A) +\frac{k^4\epsilon^4}{16}\tr(Q^{\delta}_m)^2.
\end{split}
\end{equation}
We now wish to study the second term on the right-hand side of \eqref{eq:trace.2}.

\begin{lemma}
    For $\mathcal{I}$ in \eqref{eq:mathcal.I} and $Q^{\delta}_-$, $Q^{\delta}_+$, $Q_0$ in \eqref{eq:Qdelta-}, \eqref{eq:Qdelta+} and \eqref{eq:Q0}, we have
\begin{align}
    \tr(\omega\mathcal{I}\wedge Q^{\delta}_-+Q^{\delta}_-\wedge\omega\mathcal{I})&=0,\label{eq:tr.Q.I.-}\\
   \tr(\omega\mathcal{I}\wedge Q^{\delta}_++Q^{\delta}_+\wedge\omega\mathcal{I})&=-4(1+\delta)\omega^2,\label{eq:tr.Q.I.+} \\
    \tr(\omega\mathcal{I}\wedge Q_0+Q_0\wedge\omega\mathcal{I})&=16\omega^2.\label{eq:tr.Q.I.0}
\end{align}
Hence, for $Q^{\delta}_m$ given in \eqref{eq:Q.delta.m}, we have
\begin{equation}\label{eq:tr.Q.I.delta.m}
     \tr(\omega\mathcal{I}\wedge Q^{\delta}_m+Q^{\delta}_m\wedge\omega\mathcal{I})=4(4\delta^2-(1+\delta)^2)\omega^2.
\end{equation}
\end{lemma}

\begin{proof}
We first observe that
\begin{align*}
    \omega\mathcal{I}\wedge e_0\wedge\left(\begin{array}{ccc} 0 & e^{\rm T} & Je^{\rm T}\\
    -e & 0 & 0\\
    -Je & 0 & 0\end{array}\right)+ e_0\wedge\left(\begin{array}{ccc} 0 & e^{\rm T} & Je^{\rm T}\\
    -e & 0 & 0\\
    -Je & 0 & 0\end{array}\right)\wedge\omega\mathcal{I}&=e_0\wedge\omega\wedge\left(\begin{array}{ccc} 0 & Je^{\rm T} & -e^{\rm T}\\
    Je & 0 & 0\\
    -e & 0 & 0\end{array}\right)
\end{align*}
and
\begin{align*}
    \omega\mathcal{I}\wedge e_0\wedge\left(\begin{array}{ccc} 0 & 0 & 0\\
  0 & [Je] & [e]\\
   0 & [e] & -[Je]\end{array}\right)+ e_0\wedge\left(\begin{array}{ccc} 0 & 0 & 0\\
  0 & [Je] & [e]\\
   0 & [e] & -[Je]\end{array}\right)\wedge\omega\mathcal{I}&
    =0.
\end{align*}
Given the formula \eqref{eq:Qdelta-} for $Q^{\delta}_-$ we deduce \eqref{eq:tr.Q.I.-}.

Similarly, we observe that
\begin{align*}
    \tr\bigg(\omega\mathcal{I}&\wedge e_0\wedge \left(\begin{array}{ccc} 0 & 2(e\times Je)^{\rm T} & (e\times e-Je\times Je)^{\rm T}\\
    -2(e\times Je) & 0 & 0 \\
    -(e\times e-Je\times Je) & 0 & 0\end{array}\right)\\
    &+e_0\wedge \left(\begin{array}{ccc} 0 & 2(e\times Je)^{\rm T} & (e\times e-Je\times Je)^{\rm T}\\
    -2(e\times Je) & 0 & 0 \\
    -(e\times e-Je\times Je) & 0 & 0\end{array}\right)\wedge\omega\mathcal{I}\bigg)=0.
\end{align*}
However,
\begin{align*}
\omega\mathcal{I}\wedge\left(\begin{array}{ccc} 0 & 0 & 0\\
0 & -Je\wedge Je^{\rm T} & Je\wedge e^{\rm T}\\
0 & e\wedge Je^{\rm T} & -e\wedge e^{\rm T}\end{array}\right)
&+\left(\begin{array}{ccc} 0 & 0 & 0\\
0 & -Je\wedge Je^{\rm T} & Je\wedge e^{\rm T}\\
0 & e\wedge Je^{\rm T} & -e\wedge e^{\rm T}\end{array}\right)\wedge\omega\mathcal{I}\\
&=\omega\wedge\left(\begin{array}{ccc} 0 & 0 &0 \\
0 & Je\wedge e^{\rm T}-e\wedge Je^{\rm T} & e\wedge e^{\rm T}+Je\wedge Je^{\rm T}\\
0 & -e\wedge e^{\rm T}-Je\wedge Je^{\rm T} & Je\wedge e^{\rm T}-e\wedge Je^{\rm T}\end{array}\right).
\end{align*}
Taking the trace of this equation yields
\begin{equation*}
    2\omega\wedge (-2e_1\wedge Je_1-2e_2\wedge Je_2-2e_3\wedge Je_3)=-4\omega^2.
\end{equation*}
The equation \eqref{eq:Qdelta+} for $Q^{\delta}_+$ then gives \eqref{eq:tr.Q.I.+}. 

Finally, we calculate
\begin{align*}
\frac{1}{2}\tr\Bigg(\omega\mathcal{I}&\wedge \left(\begin{array}{ccc}
        0 & 0 & 0 \\ 
        0 & -[e\times e+Je\times Je] & -2([e]\wedge [Je]-[Je]\wedge [e])\\
        0 & 2([e]\wedge [Je] -[Je]\wedge [e]) & -[e\times e+Je\times Je]
    \end{array}\right)
    \Bigg)\\
    &+\frac{1}{2}\tr\left(\left(\begin{array}{ccc}
        0 & 0 & 0 \\ 
        0 & -[e\times e+Je\times Je] & -2([e]\wedge [Je]-[Je]\wedge [e])\\
        0 & 2([e]\wedge [Je] -[Je]\wedge [e]) & -[e\times e+Je\times Je]
    \end{array}\right)\wedge\omega\mathcal{I}\right)\\
    &=
    \tr\left(
    \begin{array}{ccc} 0 & 0 & 0\\
    0 & -2\omega\wedge ([e]\wedge [Je]-[Je]\wedge [e]) & \omega\wedge [e\times e+Je\times Je] \\
    0 & -\omega\wedge [e\times e+Je\times Je] & -2\omega\wedge([e]\wedge [Je]-[Je]\wedge[e]) \end{array}
    \right)\\\
    &=-4\omega\wedge\tr ([e]\wedge [Je]-[Je]\wedge [e])=-4\omega\wedge (-4\omega)=16\omega^2
\end{align*}
by \eqref{eq:diag}.  Hence, \eqref{eq:tr.Q.I.0} holds, and equation \eqref{eq:tr.Q.I.delta.m} then immediately follows from \eqref{eq:Q.delta.m} and \eqref{eq:tr.Q.I.-}--\eqref{eq:tr.Q.I.0}.
\end{proof}

Inserting \eqref{eq:tr.Q.I.delta.m} in \eqref{eq:trace.2}, we obtain:
\begin{equation}\label{eq:trace.3}
\begin{split}
    \tr(R^2_{\theta^{\delta,k}_{\epsilon,m}  }-F_A^2)&=\frac{k^2\epsilon^4(1-\delta+m)\big(k(4\delta^2-(1+\delta)^2)-3\big)}{2}\omega^2\\
&\quad+\frac{k^2\epsilon^2}{4}\tr(F_A\wedge Q^{\delta}_m+Q^{\delta}_m\wedge F_A)+\frac{k^4\epsilon^4}{16}\tr(Q^{\delta}_m)^2.
\end{split}
\end{equation}

\subsection{Linear contribution from the \texorpdfstring{$\rG_2$}{G2} field strength}

In this subsection, we wish to analyse the term $\tr(F_A\wedge Q^{\delta}_m+Q^{\delta}_m\wedge F_A)$ from \eqref{eq:trace.3}.  

\begin{lemma}\label{lem:Qdelta.pm.FA}
For $Q^\delta_-$ in \eqref{eq:Qdelta-} and $Q^{\delta}_+$ in \eqref{eq:Qdelta+} we have
\begin{equation*}
    \tr (F_A\wedge Q^{\delta}_-+Q^{\delta}_-\wedge F_A)=0\qandq \tr (F_A\wedge Q^{\delta}_++Q^{\delta}_+\wedge F_A)=0
\end{equation*}
\end{lemma}

\begin{proof}
We see, from \eqref{eq: F ^ e and Je},  that
\begin{align*}
F_A\wedge e_0\wedge\left(\begin{array}{ccc} 0 & e^{\rm T} & Je^{\rm T}\\
    -e & 0 & 0\\
    -Je & 0 & 0\end{array}\right)&+e_0\wedge\left(\begin{array}{ccc} 0 & e^{\rm T} & Je^{\rm T}\\
    -e & 0 & 0\\
    -Je & 0 & 0\end{array}\right)\wedge F_A\\
    &=e_0\wedge \left(\begin{array}{ccc} 0 & e^{\rm T}\wedge \alpha-Je^{\rm T}\wedge\beta & e^{\rm T}\wedge\beta+Je^{\rm T}\wedge \alpha \\
    -\alpha\wedge e-\beta\wedge Je & 0  & 0 \\
    \beta\wedge e-\alpha\wedge Je & 0 & 0 \end{array}\right)=0.    
\end{align*}
We may also compute
\begin{align*}
&\tr(F_A\wedge  e_0\wedge\left(\begin{array}{ccc} 0 & 0 & 0\\
  0 & [Je] & [e]\\
   0 & [e] & -[Je]\end{array}\right)+ e_0\wedge\left(\begin{array}{ccc} 0 & 0 & 0\\
  0 & [Je] & [e]\\
   0 & [e] & -[Je]\end{array}\right)\wedge F_A)\\
   &=e_0\wedge \tr\left(\begin{array}{ccc} 0 & 0 & 0 \\
   0 & \alpha\wedge [Je]+\beta\wedge [e]+[Je]\wedge\alpha-[e]\wedge\beta & \alpha\wedge [e]-\beta\wedge[Je]+[Je]\wedge\beta+[e]\wedge\alpha \\
   0 & -\beta\wedge[Je]+\alpha\wedge [e]+[e]\wedge\alpha+[Je]\wedge\beta & -\beta\wedge[e]-\alpha\wedge[Je] +[e]\wedge\beta-[Je]\wedge\alpha \end{array}\right)=0.
\end{align*}
The first result now follows from \eqref{eq:Qdelta-}.

For the second equation, we clearly have 
\begin{align*}
    &\tr(F_A\wedge
    \left(\begin{array}{ccc} 
        0 & 2(e\times Je)^{\rm T} & (e\times e-Je\times Je)^{\rm T} \\
        -2(e\times Je) & 0 &0\\
        -(e\times e-Je\times Je) & 0 & 0
    \end{array}\right)\\
    &\qquad+
    \left(\begin{array}{ccc} 
        0 & 2(e\times Je)^{\rm T} & (e\times e-Je\times Je)^{\rm T} \\
        -2(e\times Je) & 0 &0\\
        -(e\times e-Je\times Je) & 0 & 0
    \end{array}\right)
    \wedge F_A)=0
\end{align*}
since the matrix the trace of which we are taking has no entries along the diagonal.  On the other hand, if we consider
\begin{align*}
&F_A\wedge\left(\begin{array}{ccc} 0 & 0 & 0\\
0 & -Je\wedge Je^{\rm T} & Je\wedge e^{\rm T}\\
0 & e\wedge Je^{\rm T} & -e\wedge e^{\rm T}\end{array}\right)
+\left(\begin{array}{ccc} 0 & 0 & 0\\
0 & -Je\wedge Je^{\rm T} & Je\wedge e^{\rm T}\\
0 & e\wedge Je^{\rm T} & -e\wedge e^{\rm T}\end{array}\right)\wedge F_A,
\end{align*}
we find that the only entries which are not trivially zero are
\begin{gather*}
    (-\alpha\wedge Je+\beta\wedge e)\wedge Je^{\rm T}-Je\wedge (Je^{\rm T}\wedge \alpha+e^{\rm T}\wedge\beta),\\
  (\alpha\wedge Je-\beta\wedge e)\wedge e^{\rm T} -Je\wedge(Je^{\rm T}\wedge\beta-e^{\rm T}\wedge\alpha),\\
  (\beta\wedge Je+\alpha\wedge e)\wedge Je^{\rm T}+e\wedge (Je^{\rm T}\wedge\alpha+e^{\rm T}\wedge\beta),\\
  -(\beta\wedge Je-\alpha\wedge e)\wedge e^{\rm T}+e\wedge (Je^{\rm T}\wedge\beta-e^{\rm T}\wedge\alpha),
\end{gather*}
yet these also vanish, by \eqref{eq: F ^ e and Je}.  Using \eqref{eq:Qdelta+} completes the result.
\end{proof}

From Lemma \ref{lem:Qdelta.pm.FA} we deduce that
$$\tr (F_A\wedge Q^{\delta}_m+Q^{\delta}_m\wedge F_A)=\delta^2\tr(F_A\wedge Q_0+Q_0\wedge F_A).$$
We conclude this section by studying this final term.

\begin{lemma}
\label{lem:tr.FA.Q0}
    For $Q_0$ in \eqref{eq:Q0}, we have 
    $$
    \tr (F_A\wedge Q_0+Q_0\wedge F_A)=0.
    $$
\end{lemma}
\begin{proof}
We first see that
\begin{align*}
    \tr (F_A\wedge Q_0)&=\tr(F_A\wedge  \left(\begin{array}{ccc} 
        0 & 0 & 0 \\ 
        0 & -[e\times e+Je\times Je] & -2([e]\wedge [Je]-[Je]\wedge [e])\\
        0 & 2([e]\wedge [Je] -[Je]\wedge [e]) & -[e\times e+Je\times Je]
    \end{array}\right))\\
    &=2\tr(-\alpha\wedge [e\times e+Je\times Je]+2\beta\wedge ([e]\wedge [Je]-[Je]\wedge [e]),
\end{align*}
and
\begin{align*}
    \tr (Q_0\wedge F_A)&=\tr(\left(\begin{array}{ccc} 
        0 & 0 & 0 \\ 
        0 & -[e\times e+Je\times Je] & -2([e]\wedge [Je]-[Je]\wedge [e])\\
        0 & 2([e]\wedge [Je] -[Je]\wedge [e]) & -[e\times e+Je\times Je]
    \end{array}\right)\wedge F_A)\\
    &=2\tr(-[e\times e+Je\times Je]\wedge \alpha +2([e]\wedge [Je]-[Je]\wedge [e])\wedge\beta).
\end{align*}
Hence,
$$\tr (F_A\wedge Q_0+Q_0\wedge F_A)=4\tr(-\alpha\wedge [e\times e+Je\times Je]+2\beta([e]\wedge [Je]-[Je]\wedge [e]).$$
Using Lemma \ref{lem: computational rules} we find that
\begin{gather*}
    [e\times e+Je\times Je]=2e\wedge e^{\rm T}+2Je\wedge Je^{\rm T},\\
    [e]\wedge [Je]-[Je]\wedge [e]=e\wedge Je^{\rm T}-Je\wedge e^{\rm T}-2\omega I.
\end{gather*}
Therefore,
$$
\tr (F_A\wedge Q_0+Q_0\wedge F_A)=8\tr\big(-(\alpha\wedge e+\beta\wedge Je)\wedge e^{\rm T}-(\alpha\wedge Je-\beta\wedge e)\wedge Je^{\rm T}\big)-16\omega\wedge\tr\beta=0
$$
by \eqref{eq: F ^ e and Je} and the fact that $\beta$ is traceless.
\end{proof}

By \eqref{eq:trace.3} and Lemmas \ref{lem:Qdelta.pm.FA} and \ref{lem:tr.FA.Q0} we obtain, for $\theta=\theta^{\delta,k}\epm$,
\begin{equation}
\label{eq:trace.4}
\begin{split}
    \tr(R_\theta^2-F_A^2) =\frac{k^2\epsilon^4(1-\delta+m)\big(k(4\delta^2-(1+\delta)^2)-3\big)}{2}\omega^2 +\frac{k^4\epsilon^4}{16}\tr(Q^{\delta}_m)^2.
\end{split}
\end{equation}

\subsection{The nonlinear contribution $\tr(Q^{\delta}_m)^2$}

We now wish to compute the term $\tr(Q^{\delta}_m)^2$ in \eqref{eq:trace.4}, to complete our analysis of the difference in the traces of the squares of the curvatures of $\theta^{\delta,k}_{\epsilon,m} $ and $A$.
We begin with the ``square terms'' in $(Q^\delta_m)^2$.

\begin{lemma}\label{lem:nonlinear.1}
For $Q^{\delta}_-$, $Q^{\delta}_+$, $Q_0$ in \eqref{eq:Qdelta-}--\eqref{eq:Q0} we have
\begin{gather*}
    \tr(Q^{\delta}_-)^2=0,\quad \tr(Q^{\delta}_+)^2=-8\delta^2\omega^2,\quad \tr(Q_0)^2=0. 
\end{gather*}
\end{lemma}

\begin{proof}
Since $Q^{\delta}_-=e_0\wedge Q$ for some matrix of 1-forms, we see immediately that $(Q^{\delta}_-)^2=0$.

For $Q^{\delta}_+$, we note that
\begin{equation}\label{eq:Qdelta+.split}
Q^{\delta}_+=\delta\left(\begin{array}{ccc}
0     & 2(e\times Je)^{\rm T} & (e\times e-Je\times Je)^{\rm T} \\
-2(e\times Je)     & 0 & 0\\
-(e\times e-Je\times Je) & 0 & 0
\end{array}\right)+(1+\delta)\left(\begin{array}{ccc}
0     &0 & 0  \\
0     & -Je\wedge Je^{\rm T} & Je\wedge e^{\rm T}\\
0 & e\wedge Je^{\rm T} & -e\wedge e^{\rm T}
\end{array}\right).
\end{equation}
We see that, in $(Q^{\delta}_+)^2$, the cross-terms coming from the pair of matrices above will be obviously traceless, so it suffices to compute the trace of each square.
We see that
\begin{align*}
&\tr\left(\begin{array}{ccc}
0     &  2(e\times Je)^{\rm T} & (e\times e-Je\times Je)^{\rm T} \\
-2(e\times Je)     & 0 & 0\\
-(e\times e-Je\times Je) & 0 & 0
\end{array}\right)^2\\
&\qquad=-4(e\times Je)^{\rm T}\wedge (e\times Je)-(e\times e-Je\times Je)^{\rm T}\wedge (e\times e-Je\times Je).
\end{align*}
We observe that
\begin{gather*}
    4(e_2\wedge Je_3-e_3\wedge Je_2)\wedge (e_2\wedge Je_3-e_3\wedge Je_2)=8e_2\wedge Je_2\wedge e_3\wedge Je_3\\
    2(e_2\wedge e_3-Je_2\wedge Je_3)\wedge 2(e_2\wedge e_3-Je_2\wedge Je_3)=8e_2\wedge Je_2\wedge e_3\wedge Je_3
\end{gather*}
and hence
$$
\tr\left(\begin{array}{ccc}
0     &  2(e\times Je)^{\rm T} & (e\times e-Je\times Je)^{\rm T} \\
-2(e\times Je)     & 0 & 0\\
-(e\times e-Je\times Je) & 0 & 0
\end{array}\right)^2=-8\omega^2.
$$
On the other hand, 
\begin{align*}&\tr \left(\begin{array}{ccc}
0     &0 & 0  \\
0     & -Je\wedge Je^{\rm T} & Je\wedge e^{\rm T}\\
0 & e\wedge Je^{\rm T} & -e\wedge e^{\rm T}
\end{array}\right)^2=Je\wedge e^{\rm T}\wedge e\wedge Je^{\rm T}+e\wedge Je^{\rm T}\wedge Je\wedge e^{\rm T}=0.
\end{align*}
This gives the result for $\tr(Q^{\delta}_+)^2$.

From the formula \eqref{eq:Q0} for $Q_0$ we see that
\begin{align*}
    \tr(Q_0)^2=\frac{1}{2}\tr[e\times e+Je\times Je]^2-2\tr([e]\wedge [Je]-[Je]\wedge [e])^2.
\end{align*}
We then calculate
\begin{align*}
\tr[e\times e+Je\times Je]^2&=\tr\left(\begin{array}{ccc} 0 & 2e_1\wedge e_2+2Je_1\wedge Je_2 & -2e_3\wedge e_1-2Je_3\wedge Je_1 \\
-2e_1\wedge e_2-2Je_1\wedge Je_2 & 0 & 2e_2\wedge e_3+2Je_2\wedge Je_3 \\
2e_3\wedge e_1+2Je_3\wedge Je_1 & -2e_2\wedge e_3-2Je_2\wedge Je_3 & 0\end{array}\right)^2\\
&=16(e_1\wedge Je_1\wedge e_2\wedge Je_2+e_3\wedge Je_3\wedge e_1\wedge Je_1+e_2\wedge Je_2\wedge e_3\wedge Je_3)=8\omega^2
\end{align*}
and
\begin{align*}
    \tr&([e]\wedge [Je]-[Je]\wedge [e])^2\\
    &\qquad=\tr\left(\begin{array}{ccc} -2e_2\wedge Je_2-2e_3\wedge Je_3 & e_2\wedge Je_1+e_1\wedge Je_2 & e_3\wedge Je_1+e_1\wedge Je_3 \\
    e_1\wedge Je_2+e_2\wedge Je_1 & -2e_3\wedge Je_3-2e_1\wedge Je_1 & e_3\wedge Je_2+e_2\wedge Je_3\\
    e_1\wedge Je_3+e_3\wedge Je_1 & e_2\wedge Je_3+e_3\wedge Je_2 & -2e_1\wedge Je_1-2e_2\wedge Je_2\end{array}\right)^2\\
    &\qquad=8(e_2\wedge Je_2\wedge e_3\wedge Je_3+e_3\wedge Je_3\wedge e_1\wedge Je_1+e_1\wedge Je_1\wedge e_2\wedge Je_2)\\
    &\qquad\quad+4(e_1\wedge Je_2\wedge e_2\wedge Je_1+e_3\wedge Je_1\wedge e_1\wedge Je_3+e_2\wedge Je_3\wedge e_3\wedge Je_2)\\
    &\qquad =4\omega^2-2\omega^2=2\omega^2.
\end{align*}
The formula for $\tr(Q_0)^2$ then follows.
\end{proof} 

We now look at the ``cross terms'' in $(Q^{\delta}_m)^2$.

\begin{lemma}
\label{lem:nonlinear.2}
    For $Q^{\delta}_-$, $Q^{\delta}_+$ 
    in \eqref{eq:Qdelta-}--\eqref{eq:Qdelta+}, we have
\begin{align*}
    \tr(Q^{\delta}_-\wedge Q^{\delta}_++Q^{\delta}_+\wedge Q^{\delta}_-)&=0.
\end{align*}
\end{lemma}

\begin{proof}
Just as for $Q^{\delta}_+$ in \eqref{eq:Qdelta+.split} we can split $Q^{\delta}_-$ as
\begin{equation}\label{eq:Qdelta-.split}
    Q^{\delta}_-=(1+\delta)e_0\wedge \left(\begin{array}{ccc}
    0     & e^{\rm T} & Je^{\rm T} \\
    -e     &  0 & 0\\
    -Je & 0 & 0
    \end{array}\right)-2\delta e_0\wedge\left(\begin{array}{ccc} 0 & 0 & 0 \\
    0 & [Je] & [e] \\
    0 & [e] & -[Je]\end{array}\right).
\end{equation}
Hence, we can break down the calculation of $ \tr(Q^{\delta}_-\wedge Q^{\delta}_++Q^{\delta}_+\wedge Q^{\delta}_-)$ into more manageable steps.  First, we see that
\begin{align*}
&    \tr\left( e_0\wedge \left(\begin{array}{ccc}
    0     & e^{\rm T} & Je^{\rm T} \\
    -e     &  0 & 0\\
    -Je & 0 & 0
    \end{array}\right)\wedge \left(\begin{array}{ccc}
0     & 2(e\times Je)^{\rm T} & (e\times e-Je\times Je)^{\rm T} \\
-2(e\times Je)     & 0 & 0\\
-(e\times e-Je\times Je) & 0 & 0
\end{array}\right)\right)\\
&+\tr\left(\left(\begin{array}{ccc}
0     & 2(e\times Je)^{\rm T} & (e\times e-Je\times Je)^{\rm T} \\
-2(e\times Je)     & 0 & 0\\
-(e\times e-Je\times Je) & 0 & 0
\end{array}\right) \wedge e_0\wedge \left(\begin{array}{ccc}
    0     & e^{\rm T} & Je^{\rm T} \\
    -e     &  0 & 0\\
    -Je & 0 & 0
    \end{array}\right)\right)\\
    &= 2e_0\wedge \big(-2e^{\rm T}\wedge (e\times Je)-Je^{\rm T}\wedge (e\times e-Je\times Je)-2\tr(e\wedge(e\times Je)^{\rm T})-\tr(Je\wedge (e\times e-Je\times Je)^{\rm T})
    \big)\\
    &=4e_0\wedge \big(-2e^{\rm T}\wedge (e\times Je)-Je^{\rm T}\wedge (e\times e-Je\times Je\big).
\end{align*}
We observe that
\begin{align*}
    2e^{\rm T}\wedge (e\times Je)&=2e_1\wedge (e_2\wedge Je_3-e_3\wedge Je_2)+2e_2\wedge (e_3\wedge Je_1-e_1\wedge Je_3)\\
    &\qquad+2e_3\wedge (e_1\wedge Je_2-e_2\wedge Je_1)\\
    &=4\Im\Omega+4Je_1\wedge Je_2\wedge Je_3,\\
    Je^{\rm T}\wedge (e\times e-Je\times Je)&=2Je_1\wedge (e_2\wedge e_3-Je_2\wedge Je_3)+2Je_2\wedge (e_3\wedge e_1-Je_3\wedge Je_1)\\
&\qquad    +2Je_3\wedge (e_1\wedge e_2-Je_3\wedge Je_1),\\
    &=2\Im\Omega -4 Je_1\wedge Je_2\wedge Je_3
\end{align*}
and thus
\begin{align*}
4e_0\wedge \big(-2e^{\rm T}\wedge (e\times Je)-Je^{\rm T}\wedge (e\times e-Je\times Je\big)=-24 e_0\wedge \Im\Omega.
    \end{align*}
    
    Now, clearly, 
\begin{align*}
     \tr\left( e_0\wedge \left(\begin{array}{ccc}
    0     & e^{\rm T} & Je^{\rm T} \\
    -e     &  0 & 0\\
    -Je & 0 & 0
    \end{array}\right)\wedge\left(\begin{array}{ccc}
0     &0 & 0  \\
0     & -Je\wedge Je^{\rm T} & Je\wedge e^{\rm T}\\
0 & e\wedge Je^{\rm T} & -e\wedge e^{\rm T}
\end{array}\right)\right)&=0,\\
    \tr\left(e_0\wedge\left(\begin{array}{ccc} 0 & 0 & 0 \\
    0 & [Je] & [e] \\
    0 & [e] & -[Je]\end{array}\right)\wedge \left(\begin{array}{ccc}
0     & 2(e\times Je)^{\rm T} & (e\times e-Je\times Je)^{\rm T} \\
-2(e\times Je)     & 0 & 0\\
-(e\times e-Je\times Je) & 0 & 0
\end{array}\right) \right) &=0,
    \end{align*}
    so for $\tr (Q^{\delta}_-\wedge Q^{\delta}_+)$ we are simply left with computing
\begin{align*}
      &\tr\left(e_0\wedge\left(\begin{array}{ccc} 0 & 0 & 0 \\
    0 & [Je] & [e] \\
    0 & [e] & -[Je]\end{array}\right)\wedge \left(\begin{array}{ccc}
0     &0 & 0  \\
0     & -Je\wedge Je^{\rm T} & Je\wedge e^{\rm T}\\
0 & e\wedge Je^{\rm T} & -e\wedge e^{\rm T}
\end{array}\right)\right)\\
&+\tr\left(\left(\begin{array}{ccc}
0     &0 & 0  \\
0     & -Je\wedge Je^{\rm T} & Je\wedge e^{\rm T}\\
0 & e\wedge Je^{\rm T} & -e\wedge e^{\rm T}
\end{array}\right)\wedge e_0\wedge\left(\begin{array}{ccc} 0 & 0 & 0 \\
    0 & [Je] & [e] \\
    0 & [e] & -[Je]\end{array}\right) \right)\\
    &=2e_0\wedge \tr\big([Je]\wedge (e\wedge e^{\rm T}-Je\wedge Je^{\rm T})+[e]\wedge (e\wedge Je^{\rm T}+Je\wedge e^{\rm T})\big).
\end{align*}    
To conclude, we notice that
\begin{align*}
 \tr\big([Je]\wedge (e\wedge e^{\rm T}-Je\wedge Je^{\rm T})\big)&=-2Je_3\wedge e_1\wedge e_2-2Je_2\wedge e_3\wedge e_1-2Je_1\wedge e_2\wedge e_3    +6Je_1\wedge Je_2\wedge Je_3 \\
 &=-2\Im\Omega+4Je_1\wedge Je_2\wedge Je_3,\\
 \tr\big([e]\wedge (e\wedge Je^{\rm T}+Je\wedge e^{\rm T})\big)&=2e_3\wedge (e_2\wedge Je_1+Je_2\wedge e_1)+2e_2\wedge (e_3\wedge Je_1+Je_3\wedge e_1)\\
 &\quad+2e_1\wedge (e_3\wedge Je_2+Je_3\wedge e_2)\\
 &=-4\Im\Omega-4Je_1\wedge Je_2\wedge Je_3,
\end{align*}
which gives 
$$
2e_0\wedge \tr\big([Je]\wedge (e\wedge e^{\rm T}-Je\wedge Je^{\rm T})+[e]\wedge (e\wedge Je^{\rm T}+Je\wedge e^{\rm T})\big)=-12e_0\wedge\Im\Omega.
$$
Hence, as claimed,
\begin{align*}
    \tr (Q^{\delta}_-\wedge Q^{\delta}_++Q^{\delta}_+\wedge Q^{\delta}_-)
    &=(1+\delta)\delta (-24e_0\wedge \Im\Omega)-2\delta(1+\delta)(-12e_0\wedge\Im\Omega)=0. 
    \qedhere
\end{align*}
\end{proof}

\begin{lemma}
\label{lem:nonlinear.3}
    For $Q^{\delta}_-$, $Q_0$, 
    respectively in \eqref{eq:Qdelta-}, \eqref{eq:Q0},  we have
\begin{align*}
    \tr(Q^{\delta}_-\wedge Q_0+Q_0\wedge Q^{\delta}_-)&=0.
\end{align*}
\end{lemma}
\begin{proof}
    Recall the splitting \eqref{eq:Qdelta-.split}.  Since we have
\begin{align*}
    &\tr\left(   e_0\wedge
    \left(\begin{array}{ccc} 
        0 & 0 & 0 \\
        0 & [Je] & [e] \\
        0 & [e] & -[Je]
    \end{array}\right)
    \wedge 
    \left(\begin{array}{ccc} 
        0 & 0 & 0 \\ 
        0 & -[e\times e+Je\times Je] & -2([e]\wedge [Je]-[Je]\wedge [e])\\
        0 & 2([e]\wedge [Je] -[Je]\wedge [e]) & -[e\times e+Je\times Je]
    \end{array}\right)\right)\\
    &=e_0\wedge \tr(-[Je]\wedge [e\times e+Je\times Je]+2[e]\wedge ([e]\wedge [Je]-[Je]\wedge [e]))\\
    &\quad +e_0\wedge\tr(-2[e]\wedge ([e]\wedge [Je]-[Je]\wedge[e])+[Je]\wedge[e\times e+Je\times Je])\\
    &=0,
\end{align*}
the result then follows from \eqref{eq:Qdelta-.split} and \eqref{eq:Q0}.
\end{proof}

\begin{lemma}
\label{lem:nonlinear.4}
    For $Q^{\delta}_+$, $Q_0$, respectively
    in \eqref{eq:Qdelta+}, \eqref{eq:Q0},  we have
\begin{align*}
   \tr(Q^{\delta}_+\wedge Q_0+Q_0\wedge Q^{\delta}_+)&= 16(1+\delta)\omega^2.
\end{align*}
\end{lemma}

\begin{proof}
Recall the splitting \eqref{eq:Qdelta+.split}.  We see that to calculate $\tr(Q^{\delta}_+\wedge Q_0)$ it suffices to compute the following:
\begin{align*}
&    \tr \left(\left(\begin{array}{ccc}
0     &0 & 0  \\
0     & -Je\wedge Je^{\rm T} & Je\wedge e^{\rm T}\\
0 & e\wedge Je^{\rm T} & -e\wedge e^{\rm T}
\end{array}\right)\wedge   \left(\begin{array}{ccc} 
        0 & 0 & 0 \\ 
        0 & -[e\times e+Je\times Je] & -2([e]\wedge [Je]-[Je]\wedge [e])\\
        0 & 2([e]\wedge [Je] -[Je]\wedge [e]) & -[e\times e+Je\times Je]
    \end{array}\right)\right) \\
    &=\tr\big((Je\wedge Je^{\rm T}+e\wedge e^{\rm T})\wedge[e\times e+Je\times Je])+2(Je\wedge e^{\rm T}-e\wedge Je^{\rm T})\wedge ([e]\wedge [Je]-[Je]\wedge [e])\big)\\
    &=2\tr(Je\wedge Je^{\rm T}+e\wedge e^{\rm T})^2-2\tr(Je\wedge e^{\rm T}-e\wedge Je^{\rm T})^2 -4\omega\wedge\tr(Je\wedge e^{\rm T}-e\wedge Je^{\rm T})
\end{align*}
by Lemma \ref{lem: computational rules}.

We first see that
\begin{align*}
    2\tr(Je\wedge Je^{\rm T}+e\wedge e^{\rm T})^2&=2(4e_1\wedge e_2\wedge Je_2\wedge Je_1+4e_3\wedge e_1\wedge Je_1\wedge Je_3+4e_2\wedge e_3\wedge Je_3\wedge Je_2)\\
    &=4\omega^2.
\end{align*}
We also see that
\begin{align*}
-2\tr(Je\wedge e^{\rm T}-e\wedge Je^{\rm T})^2&=-2\tr(Je\wedge e^{\rm T})^2-2\tr(e\wedge Je^{\rm T})^2\\
&=-2( 2Je_1\wedge e_2\wedge Je_2\wedge e_1+2Je_3\wedge e_1\wedge Je_1\wedge e_3+2Je_2\wedge e_3\wedge Je_3\wedge e_2)\\
&\quad-2(2e_1\wedge Je_2\wedge e_2\wedge Je_1+2e_3\wedge Je_1\wedge e_1\wedge Je_3+2e_2\wedge Je_3\wedge e_3\wedge Je_2)\\
&=4\omega^2
\end{align*}
and
\begin{align*}
    -4\omega\wedge\tr(Je\wedge e^{\rm T}-e\wedge Je^{\rm T})&=-4\omega\wedge (-2\omega)=8\omega^2.
\end{align*}
Hence,
\begin{align*}
&    \tr \left(\left(\begin{array}{ccc}
0     &0 & 0  \\
0     & -Je\wedge Je^{\rm T} & Je\wedge e^{\rm T}\\
0 & e\wedge Je^{\rm T} & -e\wedge e^{\rm T}
\end{array}\right)\wedge   \frac{1}{2}\left(\begin{array}{ccc} 
        0 & 0 & 0 \\ 
        0 & -[e\times e+Je\times Je] & -2([e]\wedge [Je]-[Je]\wedge [e])\\
        0 & 2([e]\wedge [Je] -[Je]\wedge [e]) & -[e\times e+Je\times Je]
    \end{array}\right)\right) \\
    &=\frac{1}{2}(4\omega^2+4\omega^2+8\omega^2)=8\omega^2.
    \end{align*}
The result then follows from \eqref{eq:Qdelta+.split} and \eqref{eq:Q0}.
\end{proof}

\begin{cor}
\label{cor:tr.Qdeltam} 
    For $Q^{\delta}_m$ in \eqref{eq:Q.delta.m}, we have
\begin{align*}
    \tr(Q^{\delta}_m)^2=8\delta^2(1+\delta)^2\omega^2.
\end{align*}
\end{cor}
\begin{proof}
From the definition of $Q^{\delta}_m$ in \eqref{eq:Q.delta.m}, using Lemmas \ref{lem:nonlinear.1}-\ref{lem:nonlinear.4}, we compute: 
\begin{align*}
    \tr(Q^{\delta}_m)^2
    &=\tr\big((1-\delta+m)Q^{\delta}_-+(1+\delta)Q^{\delta}_++\delta^2Q_0\big)^2\\
    &=(1-\delta+m)^2\tr(Q^{\delta}_-)^2 +(1+\delta)^2\tr(Q^{\delta}_+)^2 +\delta^4\tr(Q_0^2) +(1-\delta+m)(1+\delta)\tr(Q^{\delta}_-\wedge Q^{\delta}_+ +Q^{\delta}_+\wedge Q^{\delta}_-)\\
    &\quad +(1-\delta+m)\delta^2\tr(Q^{\delta}_-\wedge Q_0+Q_0\wedge Q^{\delta}_-) +(1+\delta)\delta^2\tr(Q^{\delta}_+\wedge Q_0+Q_0\wedge Q^{\delta}_+)\\
    &=-8(1+\delta)^2\delta^2\omega^2 +16(1+\delta)^2\delta^2\omega^2\\
    &=8\delta^2(1+\delta)^2\omega^2.\qedhere
\end{align*}
\end{proof}

Combining Corollary \ref{cor:tr.Qdeltam} and \eqref{eq:trace.4}, we conclude that
\begin{equation}
    \tr(R_\theta^2-F_A^2)
    =\frac{k^2\epsilon^4\big(k^2\delta^2(1+\delta)^2 +(1-\delta+m)\big(k(4\delta^2-(1+\delta)^2)-3\big)\big)}{2}\omega^2,
    \qwithq
    \theta=\theta^{\delta,k}_{\epsilon,m} .
\end{equation}

\subsection{Proof of Theorem \ref{thm: Main Theorem}}
\label{sec: proof of main theorem}

We are now in position to prove the final parts (iv) and (v) in Theorem \ref{thm: Main Theorem}.
Replacing the Chern--Simons defect \eqref{eq:trace.4}, between gauge fields $A$ and $\theta$, in the heterotic Bianchi identity \eqref{eq:anomaly.eps.delta.m.k}, we obtain
\begin{equation}
    -\epsilon^2\omega^2=-\frac{\alpha'}{4}\frac{k^2\epsilon^4\big(k^2\delta^2(1+\delta)^2+(1-\delta+m)\big(k(4\delta^2-(1+\delta)^2)-3\big)\big)}{2}\omega^2.
\end{equation}
Hence, there is a solution for $\alpha'>0$ if, and only if,
\begin{equation}
    k^2\big(k^2\delta^2(1+\delta)^2+(1-\delta+m)\big(k(4\delta^2-(1+\delta)^2)-3\big)\big)>0,
\end{equation}
in which case
\begin{equation}
    \alpha'=\frac{8}{k^2\epsilon^2\big(k^2\delta^2(1+\delta)^2+(1-\delta+m)\big(k(4\delta^2-(1+\delta)^2)-3\big)\big)}.
\end{equation}
We deduce the following constraints to have \color{black} an approximate solution to the heterotic $\rG_2$ system, in the sense that all of the conditions in Definition~\ref{dfn:heterotic.G2} for the heterotic $\rG_2$ system are satisfied except that we only require that $\theta$ be a $\rG_2$-instanton to order $O(\alpha')^2$ as in \eqref{eq:order.alpha}\color{black}:

\begin{prop}
There is an approximate solution to the heterotic $\rG_2$ system  if and only if
\begin{equation}
\label{eq:lambda_0}
    \lambda_0:=k^2\epsilon^2\big(k^2\delta^2(1+\delta)^2+(1-\delta+m)\big(k(4\delta^2-(1+\delta)^2)-3\big)\big)>0
\end{equation}
is large so that
\begin{equation}
    \alpha'=\frac{8}{\lambda_0}>0
\end{equation}
is small and the terms in the $\rG_2$-instanton condition \eqref{eq:G2.inst.delta.eps.k},
\begin{equation} 
    \lambda_1:=\frac{k\epsilon^2\big(6(1-\delta+m) +k(1-\delta)(1+3\delta)\big)}{4},
    \quad 
    \lambda_2:=\frac{k^2\epsilon^2
    }{4}(1+m-5\delta)(1+\delta), 
    \quad 
    \lambda_3:=\frac{k^2\epsilon^2}{4} (\delta^2-2(2+m)\delta-1)  
    \end{equation}
are all $O(\alpha')^2$, so that \eqref{eq:order.alpha} is satisfied.
\end{prop}

Inspecting \eqref{eq:lambda_0}, there are at least three manifest Ansätze for this asymptotic regime, all of which satisfy items (i)--(v) of Theorem \ref{thm: Main Theorem}:

\begin{case}
$1-\delta+m=0$ and $\delta\neq0,-1$: 
$$
    \alpha' =\frac{8}{\delta^2(1+\delta)^2} \frac{1}{\epsilon^2k^4},
\quad
    \lambda_1=\frac{(1-\delta)(1+3\delta)}{4}k^2\epsilon^2,
\quad 
    \lambda_2=-\delta(1+\delta)k^2\epsilon^2, 
\quad 
    \lambda_3=-\frac{(\delta+1)^2}{4} k^2\epsilon^2.  
$$
In order to have $k^2\epsilon^2=O(\alpha')^2$, we may take, for instance,
$$
k^2=\frac{1}{(\alpha')^3}\qandq \epsilon^2=\frac{8}{\delta^2(1+\delta)^2}(\alpha')^5,\qwithq 
\delta\neq0,-1 \qandq 
m=\delta-1,
$$
which is physically meaningful with $\epsilon\ll 1$ and $k\gg 1$.
\end{case}

\begin{case}
$\delta=0$ and $(1+m)(k+3)<0$: 
$$
    \alpha' =-\frac{8}{(1+m)(1+\frac{3}{k})} \frac{1}{\epsilon^2k^3},
\quad
    \lambda_1=\frac{\big(1+\frac{6(1+m)}{k}\big)}{4} k^2\epsilon^2,
\quad 
    \lambda_2=\frac{1+m}{4}k^2\epsilon^2, 
\quad 
    \lambda_3=-\frac{1}{4}k^2\epsilon^2. 
$$
In order to have $k\epsilon^2=O(\alpha')^2$ and $k^2\epsilon^2=O(\alpha')^2$, we may take, for instance,
$$
k=\frac{1}{(\alpha')^3}\qandq \epsilon^2=\frac{8}{(1+m)(1+3(\alpha')^3)}(\alpha')^8,\qwithq 
m<-1,
$$
which is physically meaningful with $\epsilon\ll 1$ and $k\gg 1$.
\end{case}

\begin{case}
$\delta=-1$ and $(2+m)(4k-3)>0$: 
$$
    \alpha' =\frac{8}{(2+m)(4-\frac{3}{k})} \frac{1}{\epsilon^2k^3},
\quad
    \lambda_1=\left(\frac{3(2+m)}{2k}-1\right) k^2\epsilon^2,
\quad 
    \lambda_2=0, 
\quad 
    \lambda_3=-\frac{2+m}{2}k^2\epsilon^2. 
$$
In order to have $k\epsilon^2=O(\alpha')^2$ and $k^2\epsilon^2=O(\alpha')^2$, we may take, for instance,
$$
k=\frac{1}{(\alpha')^3}\qandq \epsilon^2=\frac{8}{(2+m)(4-3(\alpha')^3)}(\alpha')^8,\qwithq 
m>-2,
$$
which is physically meaningful with $\epsilon\ll 1$ and $k\gg 1$.
\end{case}

\noindent NB.: Several other solution regimes are possible, in particular one may adjust the choices of $m$ and $\delta$ to the string scale
$\alpha'$ itself. Furthermore, it should be noted that  the asymptotic properties of $\epsilon(\alpha')$ and $k(\alpha')$ as $\alpha'\to0$ are a \emph{consequence} of the heterotic Bianchi identity \eqref{eq:anomaly.eps.delta.m.k} and the $\rG_2$-instanton condition \eqref{eq:G2.inst.delta.eps.k} `up to $O(\alpha')^2$ terms', and therefore \emph{not a choice} imposed on the Ansatz.

\newpage
\appendix

\section{Covariant matrix operations}
\label{sec: appendix - matrix operations}

\begin{definition}
\label{dfn:box.vector}
    For a $3\times 1$ vector $a$, we define $[a]$ by
\begin{equation}
\label{eq:box.vector}
    \left[\left(\begin{array}{c}     a_1\\ a_2\\ a_3 \end{array}\right)\right]
    =
    \left(\begin{array}{ccc} 
        0 & a_3 & -a_2\\
        -a_3 & 0 & a_1\\ 
        a_2 &-a_1 & 0 
    \end{array}\right).
\end{equation}
\end{definition}
This leads us to the following definition and lemma.

\begin{definition}
\label{dfn:cross}
Let
\begin{equation*}
    a=\left(\begin{array}{c}
         a_1  \\
         a_2 \\
         a_3 \end{array}\right)\quad \text{and}\quad 
          b=\left(\begin{array}{c}
         b_1  \\
         b_2 \\
         b_3 \end{array}\right)
\end{equation*}
    be vectors of $1$-forms and define
\begin{equation}
    a\times b 
    =\left(\begin{array}{c} a_2\wedge b_3-a_3\wedge b_2\\
    a_3\wedge b_1-a_1\wedge b_3\\
    a_1\wedge b_2-a_2\wedge b_1
    \end{array}\right).
    \end{equation}
Notice that
\begin{equation}
    b\times a=a\times b.
\end{equation}
\end{definition}

\begin{lemma}
\label{lem: computational rules}
  Let $a$ and $b$ be $3\times 1$ vectors of $1$-forms. 
  Then
\begin{align}
    [a]\wedge b &=-a\times b,\\
    a^{\rm T}\wedge [b]&=-(a\times b)^{\rm T},\\
    [a]\wedge [b]+[b]\wedge [a]
    &=- [a\times b],\\
    [a]\wedge [b]-[b]\wedge [a]
    &=a\wedge b^\rT - b\wedge a^\rT -2I\otimes \sum_{j=1}^3 a_j\wedge b_j.
\end{align}
In particular,
\begin{equation}
[a]\wedge [a] =-a\wedge a^\rT = -\frac{1}{2}[a\times a].
\end{equation}
\end{lemma}

\begin{proof}
We first see that
\begin{align*}
    [a]\wedge b&=  \left(\begin{array}{ccc} 
        0 & a_3 & -a_2\\
        -a_3 & 0 & a_1\\ 
        a_2 &-a_1 & 0 
    \end{array}\right)\wedge \left(\begin{array}{c}
         b_1  \\
         b_2 \\
         b_3 
    \end{array}\right)
    =
    \left(\begin{array}{c}
        a_3\wedge b_2-a_2\wedge b_3\\
        a_1\wedge b_3-a_3\wedge b_1\\
        a_2\wedge b_1-a_1\wedge b_2
    \end{array}\right)\\
    &=-a\times b
\end{align*}
by Definition \ref{dfn:cross}.  Similarly, 
\begin{align*}
    a^{\rm T}\wedge [b]&=(\begin{array}{ccc} a_1 & a_2 & a_3\end{array})\wedge \left(\begin{array}{ccc} 
        0 & b_3 & -b_2\\
        -b_3 & 0 & b_1\\ 
        b_2 &-b_1 & 0 
    \end{array}\right)\\
    &=(\begin{array}{ccc} -a_2\wedge b_3+a_3\wedge b_2 &
    -a_3\wedge b_1+a_1\wedge b_3 &
    -a_1\wedge b_2+a_2\wedge b_1 \end{array})\\
   & =-(a\times b)^{\rm T}.
\end{align*}

From Definition \ref{dfn:cross} we see that
\begin{equation*}
    [a\times b]=  \left(\begin{array}{ccc} 
        0 & a_1\wedge b_2-a_2\wedge b_1 & a_1\wedge b_3-a_3\wedge b_1 \\
        a_2\wedge b_1-a_1\wedge b_2 & 0 & a_2\wedge b_3-a_3\wedge b_2\\
        a_3\wedge b_1-a_1\wedge b_3 & a_3\wedge b_2-a_2\wedge b_3 & 0
    \end{array}\right).
\end{equation*}
On the other hand,
\begin{align*}
    [a]\wedge [b]&=
    \left(\begin{array}{ccc} 
        0 & a_3 & -a_2\\
        -a_3 & 0 & a_1\\ 
        a_2 &-a_1 & 0 
    \end{array}\right)\wedge \left(\begin{array}{ccc} 
        0 & b_3 & -b_2\\
        -b_3 & 0 & b_1\\ 
        b_2 &-b_1 & 0
    \end{array}\right)\\
    &=
    \left(\begin{array}{ccc}           -a_2\wedge b_2 - a_3\wedge     b_3 & a_2\wedge b_1 &          a_3\wedge b_1\\
        a_1\wedge b_2 & -a_3\wedge b_3-a_1\wedge b_1 & a_3\wedge b_2\\ a_1\wedge b_3 &a_2\wedge b_3 & -a_1\wedge b_1-a_2\wedge b_2
    \end{array}\right)\\
    &= -b\wedge a^\rT - I\otimes \sum_{j=1}^3 a_j\wedge b_j.
\end{align*}
and
\begin{align*}
    [b]\wedge [a] &=\left(\begin{array}{ccc}         -b_2\wedge a_2-b_3\wedge       a_3 & b_2\wedge a_1 &          b_3\wedge a_1\\
        b_1\wedge a_2 & -b_3\wedge a_3-b_1\wedge a_1 & b_3\wedge a_2\\ b_1\wedge a_3 & b_2\wedge a_3 & -b_1\wedge a_1-b_2\wedge a_2
    \end{array}\right)\\
    &=
    \left(\begin{array}{ccc}           a_2\wedge b_2 + a_3\wedge      b_3 & -a_1\wedge b_2 &         -a_1\wedge b_3\\
        -a_2\wedge b_1 & a_3\wedge b_3+a_1\wedge b_1 & -a_2\wedge b_3\\ -a_3\wedge b_1 & -a_3\wedge b_2 & a_1\wedge b_1 + a_2\wedge b_2
    \end{array}\right)\\
    &= - a \wedge b^\rT + I\otimes \sum_{j=1}^3 a_j\wedge b_j.
    \qedhere
\end{align*}
\end{proof}


\bibliography{Bibliografia-2021-10}

\end{document}